\documentclass{article}
\usepackage{graphicx} % Required for inserting images

\usepackage{geometry}
\usepackage{graphicx}
\usepackage{amsmath,amssymb,amsthm,amsfonts}
\usepackage[utf8]{inputenc}
\usepackage[greek,english]{babel}
\usepackage{enumerate}
\usepackage{hyperref}
\usepackage[active]{srcltx}
\numberwithin{equation}{section}
\usepackage[T1]{fontenc}
\usepackage{color}

\newtheorem{theo}{Theorem}

\newtheorem{defi}{Definition}[section]

\newtheorem{prop}{Proposition}[section]
\newtheorem{rmk}{Remark}[section]

\newcommand{\R}{\mathbb{R}}

\title{Traveling waves of a reaction-diffusion system with partially coupled diffusion}
\author{T. Giletti\footnote{Laboratoire de Math\'{e}matiques Blaise Pascal, UMR 6620, Universit\'{e} Clermont-Auvergne.}, H. Izuhara\footnote{Faculty of Engineering, University of Miyazaki.}, H. Monobe\footnote{Department of Mathematics, Graduate School of Science, Osaka Metropolitan University.}}
\date{}

\begin{document}

\maketitle 
\begin{center}
    \textit{This work is dedicated to Professor Hiroshi Matano, with the deepest gratitude for his continued mentorship and friendship.}
\end{center}
\begin{abstract}
	This work is devoted to the study of traveling wave solutions of reaction-diffusion systems, where the diffusion rate of~$v$, the second component, depends on~$u$, the first component. Such systems arise in prey-predator models, where the predator~$v$ is actively hunting its prey~$u$, or in epidemiological models where the disease induces erratic behavior, for example rabies. This results in a quasilinear coupling in the highest-order term of the equation for~$v$. From the theoretical point of view, we fully classify traveling wave solutions when the first component does not diffuse, in which case the problem can be reduced to a nonlinear scalar equation. The case where both components diffuse is investigated numerically. 
\end{abstract}

\section{Introduction} 

In this paper, we will be interested in traveling wave solutions of a class of two-component reaction-diffusion systems which include coupling in a diffusion term. More precisely, we consider
\begin{equation}\label{eq:cross_sys_zero}\tag{$P$}
	\left\{
	\begin{array}{l}
		\partial_t u =   -  \beta uv , \vspace{3pt} \\
		\partial_t v = \partial_x^2 ( g(u) v ) +  h (u)v ,
	\end{array}
	\right. \quad t >0, \ x \in \mathbb{R}.
\end{equation}
The particularity of this system and novelty of this work is the fact that the diffusion in the $v$-equation depends on the other component~$u$. Throughout this work, we will assume that the parameter~$\beta$ is a positive constant, and that $g,h : \mathbb{R} \to \mathbb{R}$
are smooth functions. 

If furthermore
$$g(u) = g_1 , \quad h (u) = \gamma \beta u  - \delta ,$$
where $g_1, \gamma , \delta$ are positive constants, then we recover a standard and well-studied model from epidemiology. In such context,~$u$ denotes the density of susceptible and~$v$ the density of infected. Parameters~$\beta$, $\gamma$ and~$\delta$ then stand, respectively, for the transmission rate, conversion rate, and infected death rate. Lastly,~$g_1$ is the motility rate of the infected, which can be related at the individual scale to a random Brownian type motion. This model has been proposed in~\cite{Kallen84,Kallen85} for rabies, due to the erratic behavior that the disease induces on the infected individuals. Indeed, in this model it is assumed that the susceptible healthy individuals do not move, or at least only on a longer time scale or on a smaller spatial scale than the infected. This is consistent with observations and explained in~\cite{Murray86}. However, such an assumption seems reasonable in a wider range of models, as was already pointed out in~\cite{Britton91}. One may for instance think of fungal diseases or tree affecting pests, in which case, the function~$u$ would refer to the density of (mostly motionless) trees, and~$v$ to the density of fungus, bacteria or vector carrying the disease. 

Beyond epidemiology, one may also have in mind ``prey-predator'' models involving for instance an unmoving vegetal species (the prey) together with a moving animal herbivorous species (which we still refer to as the predator). We also point out that~\eqref{eq:cross_sys_zero} does not include any growth or death term in the equation for~$u$ apart from the coupled term. As an epidemiological model, this means that the natural life cycle of the susceptible should be much longer than the dynamics of the disease, i.e. transmission, incubation and death/recovery rates. This is again consistent with many of the aforementioned applications.

When the function~$g$ is not constant, this system exhibits a structure reminiscent of cross-diffusion, in the sense that the diffusion of one of the components depends non-trivially on the other. We will see later on that, thanks to the fact that the other species does not diffuse, system~\eqref{eq:cross_sys_zero} can be reduced to a scalar equation with nonlinear diffusion on the accumulated density of infected. From the modeling point of view, this additional coupling term in the diffusive part means that the infected/predator individuals move differently depending on the density of susceptible/prey. Typically we will assume that $g$ is nonincreasing, which means that the infected/predator moves more in the absence of susceptible/prey. In other words, the former actively seek, consciously or not, the latter. This especially makes sense in a prey-predator context, in which case~$v$ would indeed be predating on~$u$. Still, it may also be applicable e.g. in the modeling of rabies where the restlessness of infected individuals may translate as an increased motility mostly in the absence of other individuals, but rather in a more aggressive behavior in the presence of other individuals. A typical example we will consider is
\begin{equation}\label{def:example1}
	g(u) = g_1 - g_2u ,\qquad h(u) = \gamma \beta u - \delta,
\end{equation}
where $g_1,g_2$ are positive constants, which is a natural generalization of the rabies model with constant diffusion considered in the aforementioned literature. While it seems less relevant to the applications presented here, as far as the mathematical analysis is concerned, our methods still apply to the case when $g$ is not nonincreasing.

Finally, there is a large literature on the mathematical study of cross-diffusion systems and their applications in biology, originating in particular from the seminal work of Shigesada-Kawasaki-Teramoto~\cite{SKT}. The property that $g(u)$ is nonincreasing is often referred to as density-suppressed motility, and it has been used in models describing bacterial colony patterns~\cite{JKW}. Furthermore, in contrast to the SKT model, the nonincreasing nature of $g(u)$ represents movement toward the attractant $u$, and hence such terms have been employed in aggregation models~\cite{FIMU}. However, most of the literature seems to be devoted to bounded domains: well-posedness and convergence to equilibria~\cite{Desvillettes14,Desvillettes15,Lepoutre}, pattern formation~\cite{Lepoutre_pattern,IM2008,LOU199679}. In contrast, as far as we know, there seems to be few studies on cross-diffusion systems in the context of spatial ecological or epidemiological invasions in an unbounded domain. There are works on the existence of traveling waves for Keller-Segel systems from chemotaxis~\cite{Calvez,KS,Lui,Nadin}, where coupling typically occurs in lower order terms. Still, we refer in particular to~\cite{KMY} for a system quite similar to ours, though its authors show the existence of traveling waves by an approach (ODE techniques and center manifold theory) very different from ours. Therefore, we believe that even the simple situation considered here may significantly improve our understanding of how coupling in the diffusion may impact the propagating properties of solutions.\\

Let us now turn to the statements of our main results. The first one is concerned with the well-posedness of the Cauchy problem associated with~\eqref{eq:cross_sys_zero}. It also includes some preliminary result on the large time behavior of solutions.
\begin{theo}\label{thm:cauchy}
	Let initial data $u_0,v_0$ be such that
	$$u_0, v_0 \geq 0, \quad u_0, v_0 \in C^{2,\alpha} (\mathbb{R};\mathbb{R}) \cap W^{2,\infty} (\mathbb{R};\mathbb{R}),$$
	for some $\alpha \in (0,1)$. Assume also that
	\begin{equation}\label{eq:ass_g} 
		\min_{0 \leq u \leq  \sup u_0} g  (u) >0 .
	\end{equation}
	Then the Cauchy problem associated with~\eqref{eq:cross_sys_zero} and initial data $u_0, v_0$ admits a unique classical, nonnegative and global-in-time solution~$(u,v)$.
	
    Furthermore, if $h(0) < 0$,
	then this solution satisfies
	\begin{align*}
		\lim_{t \to \infty}v(t,x)=0 \quad \mbox{in} \ \mathbb{R},  \quad 
		\lim_{t \to \infty}u(t,x)
		\begin{cases}
			>0 & {\rm if} \ x \in \{x \ |\ u_0(x) >0 \}, \\
			=0 & {\rm otherwise}.
		\end{cases}
	\end{align*}
\end{theo}
	We point out that, by comparison and the nonpositivity of $-\beta u v$, one expects $u$ to never exceed its initial supremum. This justifies the hypothesis~\eqref{eq:ass_g}, which ensures that the equation for~$v$ remains parabolic, and in turn the well-posedness of the Cauchy problem.
	
	The second part of Theorem~\ref{thm:cauchy} shows that the susceptible/prey persists anywhere it was initially present, at least under the assumption that $h(0)<0$; the latter inequality simply means that the density of infected/predators decreases in the absence of susceptible/prey. We will come back below to the relation between the initial and final densities.\\

In the rest of this work we will be interested in the existence of traveling wave solutions for~\eqref{eq:cross_sys_zero}, that is of special solutions depending on a moving variable $x - ct$ where $c \in \mathbb{R}$. In the special case when
$$g(u) = g_1 , \quad h (u) = \gamma \beta u  - \delta,$$
which in particular implies that there is no coupling in the diffusion, the existence of traveling fronts was already established in~\cite{Kallen84}. More recently, this result was extended to the spatially periodic case in~\cite{DucrotGiletti}, where it was further proved that solutions of the Cauchy problem converge to a traveling wave for large classes of initial data. The proofs relied on the observation, from~\cite{Britton91}, that the time integral of~$v$, or in other words the spatial distribution of the cumulated density of infected, satisfies a scalar reaction-diffusion equation. We will make an extensive use of this fact, which remains true for the more general system~\eqref{eq:cross_sys_zero}, up to the new technical difficulties arising from the coupled diffusion.

Let us first introduce a more precise definition of traveling wave solutions. For an illustration, Figure~\ref{fig:1} shows a profile of a traveling wave solution of \eqref{eq:cross_sys_zero} obtained numerically.
\begin{defi}\label{defi_tw}
A traveling wave solution of~\eqref{eq:cross_sys_zero} is an entire solution (i.e. a solution defined for all $(t,x) \in \mathbb{R}^2$) of the form $(u,v) (t,x) = (U,V) (x-ct)$, where $c \in \mathbb{R}$ is the traveling wave speed, and $U,V>0$ satisfy
$$(U,V) (+\infty) = (U_{initial},0), \qquad (U,V) (-\infty) = (U_{final}, 0),$$
for some constants $0 < U_{final} < U_{initial} < + \infty$. We will refer to $U_{initial}$ and $U_{final}$ as respectively the right and left limits.
\end{defi}

\begin{figure}[htbp]
    \centering
    \includegraphics[width=80mm]{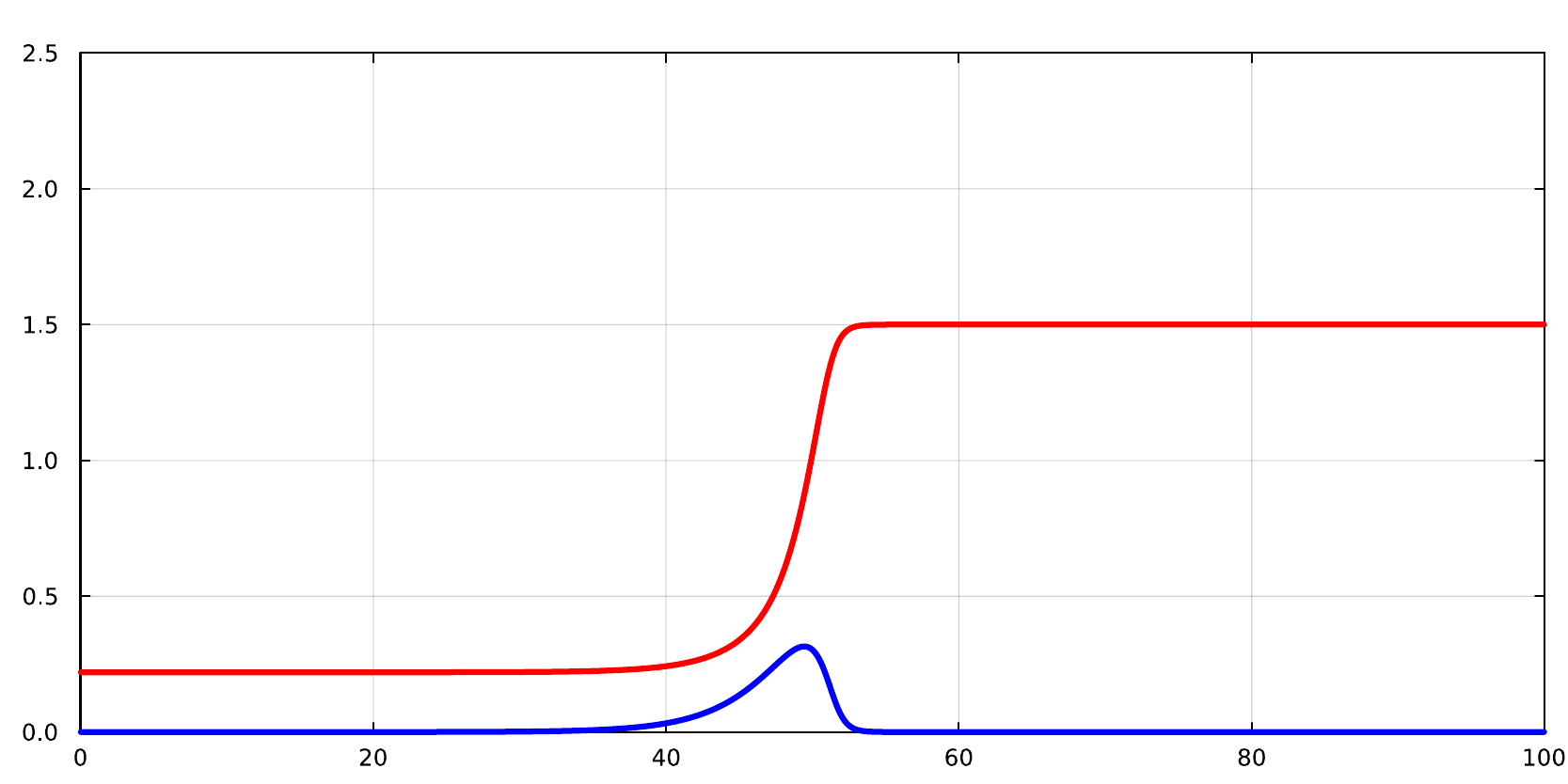}
    \caption{Traveling wave solution of \eqref{eq:cross_sys_zero} when $\beta = U_{initial} = 1.5$, $g(u) = 2 - u$ and $h(u) = \beta u v - v$. The red curve and the blue curve respectively represent $U$ and $V$. Then, the wave speed and $U_{final}$ are approximately $c=1.6328$ and $U_{final}=0.21662$.}
    	% $\delta=1.0$, $\gamma=1.0$, $\beta=1.5$, $g_1=2.0$, $g_2=1.0$ and $u_0=1.5$. The red curve and the blue curve respectively represent $U$ and $V$. Then, the wave speed and $U_\infty$ are approximately $c=1.6328$ and $U_\infty=0.21662$. }
    \label{fig:1}
\end{figure}

The following theorem is concerned with traveling wave solutions of \eqref{eq:cross_sys_zero} for general choices of the nonlinearities~$g$ and~$h$.
\begin{theo}\label{th:main_tw}
	Assume that $h$, $g$ and $U_{initial}>0$ are such that
\begin{equation}\label{hyp:u0}
	\begin{array}{l}
	h (U_{initial}) > 0 > h (0), \qquad 
%	\vspace{5pt}  
	\forall u \in [0,U_{initial}], \quad 
	h' (u ) >0, \quad g (u)>0.
\end{array}
 \end{equation}
Then there exists $c^* (U_{initial}) >0$ such that~\eqref{eq:cross_sys_zero} admits a traveling wave with speed~$c$ and right limit $U_{initial}$ if and only if $c \geq c^* (U_{initial})$.  Moreover, its left limit is 
$$U_{final} = U_{initial} e^{-\beta W_+},$$
where $W_+$ is the unique positive solution of
$$\int_0^{W_+} h (U_{initial} e^{-\beta s}) ds = 0.$$
\end{theo}
We highlight the fact that the value of the right limit~$U_{initial}$ is not prescribed here, as long as it satisfies~\eqref{hyp:u0}. Each value of~$U_{initial}$, i.e. each value of the density of susceptible/prey before the epidemic/predator outbreak, gives rise to a different family of traveling waves. On the other hand, the value of the left limit~$U_{final}$ is uniquely determined by the choice of~$U_{initial}$.

Notice that the existence and uniqueness of~$W_+$ is an elementary consequence of assumption~\eqref{hyp:u0}. The value of~$W_{+}$ is especially meaningful. Indeed, according to Theorem~\ref{th:main_tw}, it is involved in the formula for the left limit of the traveling wave, which can also be understood as the final density of susceptible/prey after the invasion of the disease/predator. But also, it will turn out to be the final cumulated density of infected/predators at any given spatial point, i.e.
$$W_+ = c \int_{-\infty}^{+\infty} V (z) dz = \int_{-\infty}^{+\infty} v(s,x)ds$$
for any $x \in \mathbb{R}$, where $v(t,x) = V (x-ct)$ is the second component of the traveling wave. 

Finally, as far as traveling waves are concerned, the situation is analogous to the monostable scalar equation, as well as to the case of the system with linear diffusion, i.e. when $g$ is constant in~\eqref{eq:cross_sys_zero}. That is, we found a family of traveling waves above some minimal critical speed, which we expect to be the invasion speed for e.g. compactly supported initial data for the predator/infected~$v$. We choose not to address this issue, though it may be handled in a similar way as done in~\cite{DucrotGiletti}. Indeed, as we already outlined, this existence result relies on showing that~\eqref{eq:cross_sys_zero} is in fact equivalent to a scalar equation. As this equation satisfies a comparison principle, its traveling waves are known to attract solutions associated with large classes of initial data (e.g. compactly supported or exponentially decaying), a fact which can be used to recover similar large-time asymptotics in the Cauchy problem associated with~\eqref{eq:cross_sys_zero}. Instead, we choose here to focus on the dependence of the traveling wave on the parameters and especially on the effect of coupling in the diffusion.

To do so, we will consider the subclass of nonlinearities~$g$ and~$h$ given in~\eqref{def:example1}. In this case, we are able to fully address how the invasion speed~$c^*$ and the final outcome~$U_{final}$ depend on the initial value~$U_{initial}$ as well as the coupling function~$g$.
\begin{theo}\label{th:main_tw2}
    Consider
	$$g (u) = g_1 - g_2 u, \quad h (u ) = \gamma \beta u - \delta,$$
	where $g_1, g_2, \beta, \gamma, \delta$ are positive constants. 
	Consider also $U_{initial}$ satisfying~\eqref{hyp:u0}, i.e.
%	as well as $U_{initial}$ such that~\eqref{hyp:u0} is satisfied, i.e.
	    $$\frac{\delta}{\gamma \beta} <  U_{initial} < \frac{g_1}{g_2}.$$
	Then the conclusions of Theorem~\ref{th:main_tw} hold true, as well as the following two statements.
	\begin{enumerate}[$(i)$]
		\item For any $c \geq c^* (U_{initial})$, the left limit of the $U$-component is given by
		$$U_{final} = U_{initial}  e^{-\beta W_+},$$
		where $W_+$ is the unique positive zero of
		$$w \mapsto \gamma U_{initial} (1 - e^{-\beta w}) - \delta w.$$
		In particular, $U_{final}$ is a decreasing function of~$U_{initial}$.
		\item The minimal speed $c^* (U_{initial})$ is given by the following formula:
$$c^* (U_{initial}) = \left\{ 
\begin{array}{ll}
	2 \sqrt{(g_1- g_2 U_{initial}) ( \gamma \beta U_{initial} - \delta)} & \quad \mbox{ if } \  U_{initial} \leq \frac{1}{2} \left(\frac{g_1}{g_2} + \frac{\delta}{\gamma \beta} \right) ,\vspace{3pt}\\
		g_1 \sqrt{\frac{\gamma \beta}{g_2}} - \delta \sqrt{\frac{\delta}{\gamma \beta}}&  \quad \mbox{ if } \ U_{initial}   \geq  \frac{1}{2} \left(\frac{g_1}{g_2} + \frac{\delta}{\gamma \beta} \right). \\
\end{array}
\right.
$$
    In particular, $c^*$ is a nondecreasing function of~$U_{initial}$.
	\end{enumerate}
\end{theo}
The first statement starts as merely a restatement of the formula for~$W_+$ from Theorem~\ref{th:main_tw} in this specific case. Still the new formula has the benefit of eventually leading to the monotonicity of~$U_{final}$ with respect to~$U_{initial}$, after a lengthy but elementary computation which we thus omit.

{The second statement handles the dependence of the speed with respect to~$U_{initial}$. Perhaps surprisingly, it provides a fully explicit formula with respect to all parameters, in the spirit of the one found for a cubic monostable nonlinearity in~\cite{HadelerRothe}. This is thanks to a generalization by~\cite{An23}, which turns out to apply here. The resulting monotonicity of~$c^*$ is striking: though increasing the value of~$U_{initial}$ increases the growth rate of the infected/predator around the invaded steady state, it also diminishes its diffusion rate. One may have expected this latter fact to slow the invasion in some parameter range. 

This monotonicity appears to be closely related to the linear determinacy of the traveling wave speed. Linear determinacy refers to the equality between~$c^*$ as defined above, and the minimal speed~$c^*_{lin}$ associated with the linearized system around the right limit $(U_{initial},0)$. Here, 
this is reduced to the scalar equation
%this boils down to the scalar equation
$$\partial_t v = g(U_{initial}) \partial_x^2 v  + h (U_{initial}) v ,$$
%{\color{magenta}
so that
\begin{equation}\label{linear_speed}
c^*_{lin} = 2 \sqrt{ g(U_{initial}) h (U_{initial})},
\end{equation}
hence, when the functions~$g$ and~$h$ are chosen as in Theorem~\ref{th:main_tw2},
$$c^*_{lin} = 2 \sqrt{(g_1- g_2 U_{initial}) ( \gamma \beta U_{initial} - \delta)}.$$
%}
Seen as a function of~$U_{initial}$, the linear speed $c^*_{lin}$ reaches a maximum when $U_{initial} = \frac{1}{2} \left(\frac{g_1}{g_2} + \frac{\delta}{\gamma \beta} \right)$. In other words, the minimal traveling wave speed~$c^*$ of the nonlinear system~\eqref{eq:cross_sys_zero} coincides with the linear speed as long as the latter is increasing with respect to~$U_{initial}$, and with its maximum otherwise (see Figure~\ref{fig:2} 
in Section \ref{sec:comp}).\\

The above discussion should be worth expanding in the more general context of cross-diffusion systems. As a matter of fact, we will end this paper with some numerics on the more general system where both species diffuse, i.e. 
\begin{equation}\label{eq:cross_sys_epsilon}\tag{$P_\varepsilon$}
	\left\{
	\begin{array}{l}
		\partial_t u =  \varepsilon \partial_x^2 u -  \beta uv , \vspace{3pt} \\
		\partial_t v = \partial_x^2 ( g(u) v ) +  h (u) v,
	\end{array}
	\right. \quad t >0, \ x \in \mathbb{R}.
 \end{equation}
We will find that the phenomena observed above are quite robust. Indeed, 
%{\color {magenta} 
numerical simulations suggest that~\eqref{eq:cross_sys_epsilon} also admits traveling waves. Moreover, if~$g$ is linear, then the minimal speed does not even depend on the diffusion parameter~$\varepsilon$. This is not entirely unexpected: indeed, the parameter~$\varepsilon$ is not involved in the linearized system around the right steady state~$(U_{initial},0)$, nor in the resulting formula~\eqref{linear_speed} for the linear speed. %this is in fact true for the linearized system around the right steady state $(U_{initial},0)$. 
Still, it is noteworthy here since, according to Theorem~\ref{th:main_tw2}, the traveling wave speed is not always linearly determined. Furthermore, we will see that, for other choices of the function~$g$, the minimal speed is in fact no longer independent of~$\varepsilon$.
 We hope to address analytically the full system, in which both species diffuse, in some future work.
 %}

\section{From~\eqref{eq:cross_sys_zero} to a scalar nonlinear equation}\label{sec:reduc}

In this section, we show how system~\eqref{eq:cross_sys_zero} can be reduced to a scalar equation, albeit at this stage with nonlinear diffusion. In the footsteps of~\cite{Britton91}, we first present a formal argument which allows us to derive that scalar equation from the full system by a straightforward integration in time. To make the argument rigorous, one may instead start from the scalar equation and show that a time derivative is well-defined and satisfies~\eqref{eq:cross_sys_zero}. This approach will eventually lead to both the well-posedness of the Cauchy problem, and the existence of traveling wave solutions.

Consider $t_0 \in \mathbb{R}$ and $(u,v)$ a solution of~\eqref{eq:cross_sys_zero} for any time $t > t_0$. Notice that
$$u (t,x) = u(t_0,x) e^{-\int_{t_0}^t \beta v(s,x) ds}.$$
Now define
$$\tilde{v}(t,x) = \int_{t_0}^t v (s,x) ds.$$
Then, integrating in time the $v$-equation,
\begin{eqnarray*}
\partial_t \tilde{v} - v (t= t_0) &= & \int_{t_0}^t  [ \partial_x^2 ( g(u) v) + h(u) v ] ds\\
%& = &  \int_{t_0}^t \partial_x^2 ( (g (u)  v) ds  + \gamma u (t= t_0)  (1 - e^{-\beta \tilde{v}} )  - \delta \tilde{v} \\
& = &  \partial_x^2 \left(  \int_{t_0}^t g(u) v ds \right)  + \int_{t_0}^t h(u)v ds\\
& = &  \partial_x^2 \left(  \int_{t_0}^t g(u(t = t_0)  e^{-\beta \tilde{v} } ) \partial_t \tilde{v} ds \right)  + \int_{t_0}^t  h(u(t = t_0)  e^{-\beta \tilde{v} } ) \partial_t \tilde{v} ds ,
\end{eqnarray*}
and
\begin{equation}\label{eq:tildev}
\partial_t \tilde{v} - v (t=t_0) =   \partial_x^2 \left(  \int_{0}^{\tilde{v} (t,x)}  g(u(t = t_0)  e^{-\beta s} )  ds \right)  +  \int_{0}^{\tilde{v} (t,x)}  h(u(t = t_0)  e^{-\beta s } )  ds 
\end{equation}

\paragraph{Cauchy problem.} First, we pick $t_0 = 0$ and get that $\tilde{v}$ at least formally solves the scalar nonlinear and heterogeneous equation
\begin{equation}\label{eq:scalar_gen_cauchy}
	\partial_t w = \partial_x^2 (d(x,w)) + f(x,w),
\end{equation}
where $\partial_x^2$ denotes the second derivative of $x \mapsto d (x,w(t,x))$, with
\begin{equation*}
	d(x,w) := \int_0^w g (u_0 (x) e^{-\beta s}) ds,	 
\end{equation*}
and 
\begin{equation*}
	f(x,w) :=  v_0 (x) + \int_0^w h (u_0 (x) e^{-\beta s}) ds . 
\end{equation*}
Figure~\ref{fig:schematic} summarizes the process to reduce system \eqref{eq:cross_sys_zero} to a scalar equation.
\begin{figure}[htbp]
    \centering
    \includegraphics[width=130mm]{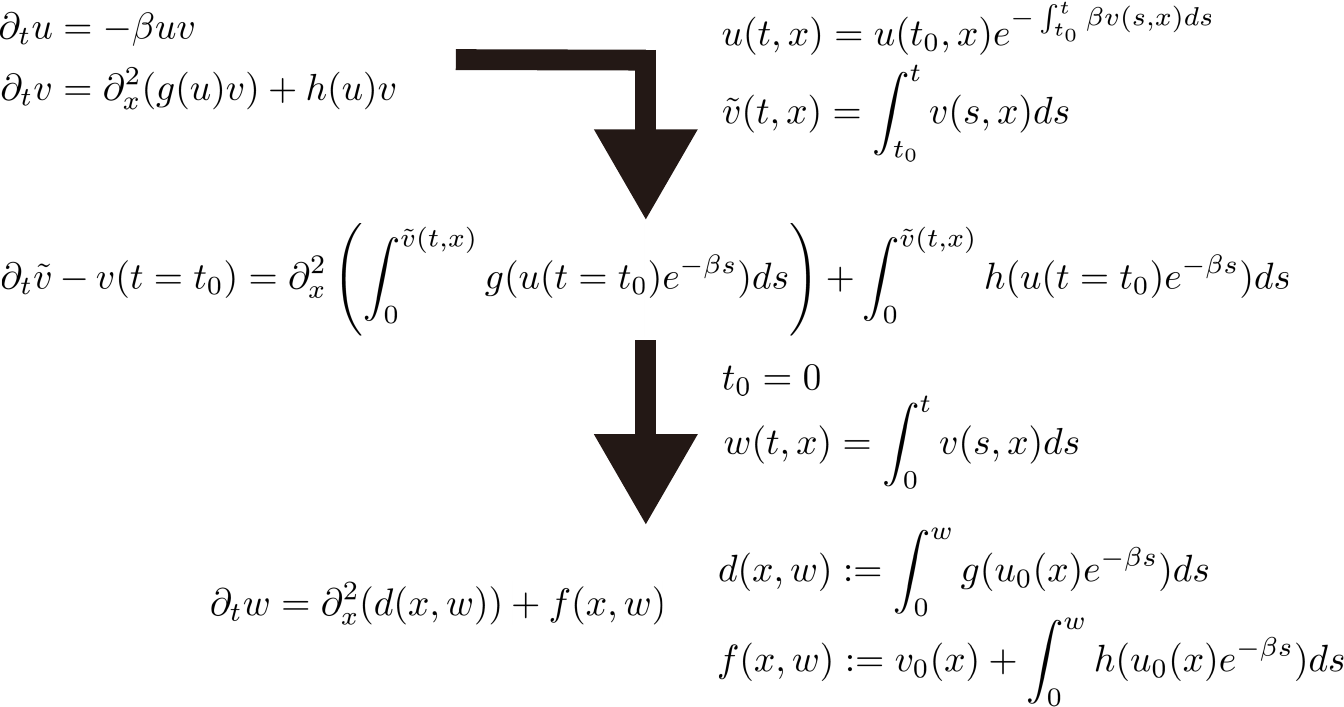}
    \caption{Schematic figure for the derivation of a scalar equation \eqref{eq:scalar_gen_cauchy} from the system \eqref{eq:cross_sys_zero}. }
    \label{fig:schematic}
\end{figure}

Notice that under the assumptions of Theorem~\ref{thm:cauchy}, the functions $d$ and $f$ are $C^{2,\alpha}$ with respect to~$x$ and smooth with respect to $w$, and also globally lipschitz with respect to $w \geq 0$. Moreover, 
$$f(x,0) =v_ 0 \geq 0,$$
and 
$$\inf_{x \in \mathbb{R}, w \geq 0} \partial_w d (x,w) = \inf_{x \in \mathbb{R}, w \geq 0} g(u_0 e^{-\beta w}) > 0.$$
In particular, we may rewrite \eqref{eq:scalar_gen_cauchy} as 
$$\partial_t w = \partial_w d (x,w) \times \partial_x^2 w   +  A ( x , w , \partial_x w),$$
with $$A (x,w, \partial_x w) := 2 \partial_x w \times \partial_{xw} d(x,w)  + \partial_{w}^2 d (x,w) \times (\partial_x w)^2 +  f(x,w)+  \partial_x^2 d (x,w) ,$$
where $\partial_x^2$ now denotes the actual second partial derivative with respect to~$x$ of the function $(x,w) \mapsto d (x,w)$.

By the standard theory of nonlinear parabolic equations~\cite{Ladyzhenskaya}, this implies that \eqref{eq:scalar_gen_cauchy}, together with the initial condition 
$$w (t=0) = 0,$$
admits a classical nonnegative solution~$w \in C^{1,2} ( [0,+\infty) \times \R )$.

Then we may define
$$u(t,x) := u (t=0,x) e^{-\beta w (t,x)} .$$
We would like to show that $(u,\partial_t w)$ is a solution of the original system~\eqref{eq:cross_sys_zero}. However, at this stage it is not clear that $\partial_t w$ possesses enough regularity. Instead, we may notice that $u \in C^{1,2} ([0,+\infty) \times \R)$, and even $u \in C^{1+ \frac{\alpha}{2}, 2 + \alpha} ([0,+\infty) \times \R)$, for some $\alpha >0$, by Schauder estimates. Due to $w \geq 0$, we also have that $0 \leq u \leq \sup u_0$, and then, by~\eqref{eq:ass_g},
$$\inf_{t \geq 0, x \in \mathbb{R}} g (u(t,x)) >0.$$
Thus
$$\partial_t v = \partial_x^2 ( g(u) v) + h(u) v$$
admits a global in time classical solution $v \in C^{1,2} ([0,+\infty) \times \R)$, which by the comparison principle is nonnegative. This in fact already concludes the proof of the well-posedness of~\eqref{eq:cross_sys_zero}. Furthermore, repeating the above integration argument, one may find that $(t,x) \mapsto  \int_0^t v(s,x) ds$ solves~\eqref{eq:scalar_gen_cauchy}, hence by uniqueness it coincides with~$w$. In other words, $\partial_t w \equiv v$ as expected.

Let us now look into the large-time behavior of this solution and end the proof of Theorem~\ref{thm:cauchy}. For any $w>0$, we have
\begin{eqnarray*}
A (x,w,0) & = & f(x,w) + \partial_x^2 d (x,w) \\
& = & v_0 (x) + \int_0^{w} h (u_0 (x) e^{-\beta s}) ds \\
& & + \int_0^{w} \left[- \beta \partial_x^2 u_0 (x) e^{-\beta s} g' (u_0 (x) e^{-\beta s}) + \beta^2 (\partial_x u_0)^2 e^{-2 \beta s} g'' (u_0 (x) e^{-\beta s} ) \right]ds \\
%& \leq & \int_0^{w} h (u_0 (x) e^{-\beta s}) ds + C \int_0^{w} e^{-\beta s} ds \\
& \leq & \|v_0\|_\infty + \int_0^{w} h (u_0 (x) e^{-\beta s}) ds + C e^{-\beta w},
\end{eqnarray*}
where~$C$ is a large positive constant coming from the boundedness of $u_0$ and its derivatives, and of~$g',g''$ on $[0, \|u_0\|_\infty]$. Thus, provided that 
$$h (0) < 0,$$
one finds some large enough positive constant~$\overline{w}$ such that
$$A (x, \overline{w}, 0) < 0.$$
Recalling that $w(t=0)=0$, we infer from the comparison principle that 
\[
w (t,x) \le \overline{w} \quad \mbox{for all } \ (t,x) \in  \mathbb{R}^+ \times \mathbb{R}.
\]
Remember that $\partial_t w = v $ is nonnegative. Thus~$w(t,x)$ converges (at least) pointwise as~$t \to +\infty$. By parabolic estimates we conclude that 
\[
\lim_{t \to \infty} v(t,x)= \lim_{t \to \infty} \partial_t w (t,x) =0, 
\]
%for all~$x \in \mathbb{R}$. It follows immediately from this result and the fact that $u(t,x) = u_0 (x) e^{-\beta w (t,x)}$ that 
\[
 \lim_{t \to \infty}u(t,x)
 \begin{cases}
 >0 & {\rm if} \ x \in \{x \ |\ u_0(x) >0 \} ,\\
 =0 & {\rm otherwise}.
\end{cases}
\]
This concludes the proof of Theorem~\ref{thm:cauchy}.

\paragraph{Traveling waves.} Next, as far as traveling wave solutions of \eqref{eq:cross_sys_zero} are concerned, we have that
$$\forall x\in \mathbb{R}, \quad \lim_{t_0 \to -\infty} ( u (t_0,x), v(t_0,x)  )  = (U_{initial},0). $$
Thus, repeating the above formal argument as above but passing instead to the limit as $t_0 \to -\infty$ in \eqref{eq:tildev}, we find that $\tilde{v}$ solves the scalar equation with nonlinear diffusion
\begin{equation}\label{eq:scalar_gen}
	\partial_t w = \partial_x^2 (d(w)) + f(w).
\end{equation}
with
\begin{equation}\label{def:dd_tw}
	d(w) := \int_0^w g (U_{initial} e^{-\beta s}) ds,	 
\end{equation}
and 
\begin{equation}\label{def:ff_tw}
	f(w) :=  \int_0^w h (U_{initial} e^{-\beta s}) ds . 
\end{equation}
Let us now invoke a traveling wave existence result for~\eqref{eq:scalar_gen}, whose proof we postpone to the next section.
\begin{prop}\label{prop:reduced}
	Consider equation~\eqref{eq:scalar_gen} with smooth $d$ and $g$. Assume that $f$ is of the monostable type, i.e.
			$$\exists K >0, \ \forall w \in (0,K), \quad f (0) = 0 = f(K) < f (w),$$
		and 
			$$\forall w \in [ 0,K] , \quad d' (w) > 0.$$
	Then there exists $c^*$ such that~\eqref{eq:scalar_gen} admits a traveling wave solution with speed $c$ if and only if $c \geq c^*$. By traveling wave with speed~$c$, here we mean a solution~$W$ of 
	\begin{equation*}
		\left( d(W)  \right)'' + c W ' +  f(W)  = 0,
	\end{equation*}
	satisfying also
	$$W(+\infty) = 0  < W (\cdot) <  W(-\infty)= K .$$
	
% 	Furthermore, for any given $c \geq c^*$, the traveling wave with speed~$c$ is unique up to shift, in the sense that, for any pair of traveling waves $W_1$ and $W_2$ with the same speed, there exists~$X \in \mathbb{R}$ such that
% 	        $$W_1 (\cdot - X) \equiv W_2.$$
\end{prop}
Let us briefly check that $d$ and $f$ defined by~\eqref{def:dd_tw}-\eqref{def:ff_tw} satisfy those assumptions. 
Notice that $f (0) = 0$ and $f' (w)  = h (U_{initial} e^{-\beta w})$. In particular, under the assumptions of Theorem~\ref{th:main_tw}, we have that $f'$ is decreasing and
$$f'(0) = h (U_{initial}) >0 >  h (0) = \lim_{w \to +\infty} f' (w).$$
It follows that $f$ is indeed of the monostable (and even, Fisher-KPP) type, where~$K>0$ such that $f(K)=0$ is uniquely defined by
$$\int_0^K h (U_{initial} e^{-\beta s}) ds,$$
i.e. it coincides with $W_+$ introduced in the statement of Theorem~\ref{th:main_tw}. Moreover, $$d' (w) = g(U_{initial} e^{-\beta w}) >0, $$ for any $w \geq 0$. Therefore the above proposition applies and \eqref{eq:scalar_gen} with $d,f$ from \eqref{def:dd_tw}-\eqref{def:ff_tw} admits a traveling wave solution with speed~$c$ if and only if $c \geq c^*$, for some $c^* >0$. 

Now take any such traveling wave solution~$W$ of~\eqref{eq:scalar_gen} with speed $c \geq c^*$. By our assumptions, the functions~$d,f$ are smooth, hence so is~$W$. Define
$$(U,V) (z) := (U_{initial} e^{-\beta W (z)}, - c W' (z)).$$
Then $(U,V)$ satisfy the following limits
$$(U,V) (+\infty) = (U_{initial}, 0), \qquad (U,V) (-\infty) = (U_{final}, 0),$$
where $U_{final} = U_{initial} e^{-\beta W_+}$. Moreover, derivating the $W$-equation, one can eventually check that~$(U,V)$ is a traveling wave solution of~\eqref{eq:cross_sys_zero} with the same speed~$c \geq c^*$.

It remains to check that there is no other (up to shift) traveling wave of~\eqref{eq:cross_sys_zero} with right limit~$U_{initial}$, and in particular that $c^*$ is the minimal traveling wave speed. To do so, consider a traveling wave $(u,v) (t,x) = (U,V) (x-ct)$, in the sense of Definition~\ref{defi_tw}. First, notice that its speed~$c$ must be positive. Indeed 
$$- cU ' (x-ct) = \partial_t u (t,x)  = - \beta u (t,x) v (t,x) < 0 .$$
Thus $c \neq 0$, and due to $ U(+\infty) = U_{initial} > U_{final} = U(-\infty)$, also $c>0$.

Then, $(U,V)$ solves the finite-dimensional ODE system
$$\left\{ 
\begin{array}{l}
-c U' =  - \beta U V , \\
- c V' = (g(U)V)'' + h (U) V,
\end{array}\right.
$$
whose linearization around the invaded steady state $(U_{initial},0)$ rewrites as 
 $$\left\{ 
 \begin{array}{l}
 	-c U' =  - \beta U_{initial} V , \\
 	- c V' = (g(U_{initial})V)'' + h (U_{initial}) V .
 	\end{array}
 	\right.
$$
Notice that, in the $V$-equation, both $g (U_{initial})$ and $h (U_{initial})$ are positive. Thus, by standard ODE theory, one finds that any solution going to $0$ as $z \to +\infty$ must do so exponentially. This allows us to define 
	$$\int_{-\infty}^t v (s ,x) d s ,$$
	for any $t,x \in \mathbb{R}$, and the above formal integration argument can now be made rigorous. More precisely, letting
	$$W (z) = \frac{1}{c} \int_z^{+\infty} V(s) ds,$$
	one may eventually find that
	$$ (d (W))'' + c W' + f(W) = 0.$$
	Still, to conclude that~$W$ is a traveling wave of \eqref{eq:scalar_gen_TW} with speed~$c$, it remains to check that it satisfies some appropriate asymptotics. Due to $V>0$, we immediately have that $W' < 0$ and $W(+\infty) = 0$. From the equation $c U' = \beta U V$, it is clear that $$U = U (+\infty) e^{-\beta W} = U_{initial} e^{-\beta W}.$$
	In particular, according to Definition~\ref{defi_tw},
	       $$\lim_{z \to -\infty} U_{initial} e^{-\beta W (z)} = U_{final} > 0,$$
	  and then the following limit exists and is positive:
	    $$ W ( -\infty) = \frac{1}{\beta} \ln \left( \frac{U_{initial}}{U_{final}} \right) .$$
	   Necessarily $f (W(-\infty)) = 0$, hence $W (-\infty) = W_+$, or equivalently~$W(-\infty) = K$ under the notation of Proposition~\ref{prop:reduced}. This confirms that $W$ is a traveling wave solution of~\eqref{eq:scalar_gen_TW} with speed~$c$, which finally implies that $c \geq c^*$. Moreover, the left limit of $(U,V)$ is necessarily equal to $U_{initial} e^{-\beta W_+}$.
	   Since the traveling wave solution of \eqref{eq:scalar_gen_TW} is unique up to translation (see Remark~\ref{rmk:uniqueness}), the above correspondence also yields uniqueness up to translation for the traveling waves of \eqref{eq:cross_sys_zero}.
	
	In other words, we have shown that our main Theorem~\ref{th:main_tw} follows from Proposition~\ref{prop:reduced}, whose proof will be done in the next section.

\section{Traveling waves for nonlinear scalar reaction-diffusion equations}\label{sec:tw}

The arguments in Section~\ref{sec:reduc} lead us to consider a scalar equation of the form~\eqref{eq:scalar_gen}, i.e.,
\begin{equation*}
	\partial_t w = \partial_x^2 ( d(w) ) + f(w).
\end{equation*}
In this section, we will collect several relevant facts about this equation and its traveling waves, as well as about the more standard semilinear monostable equation
\begin{equation}\label{eq:scalar_semiF}
	\partial_t w  = \partial_x^2 w+ F(w),
\end{equation}
which will turn out to be equivalent to \eqref{eq:scalar_gen} by letting $F(w) = d'(w) f(w)$ (see Section~\ref{sec:exist_tw} for the details). While this collection of facts, including the connection between~\eqref{eq:scalar_gen} and~\eqref{eq:scalar_semiF}, may be of independent interest, we do not claim any novelty as they can be found in particular in~\cite{An23,AW75,GildingKersner}.

Throughout this section, we will assume the following:
\begin{enumerate}[(H1)]
\item the function $s \mapsto d(s)$ is smooth and 
$$0< \inf d' \leq \sup d' < +\infty;$$
\item the function $f$ is of the monostable type, i.e. it is smooth and satisfies (up to a normalization)
$$f(0) = f(1) = 0, \quad f_{|(0,1)} >0.$$
\end{enumerate}
As far as~\eqref{eq:scalar_semiF} is concerned, we will similarly assume that
\begin{enumerate}[(H1')]
	\item the function $F$ is of the monostable type, i.e. it is smooth and 
$$F(0) = F(1) = 0, \quad F_{|(0,1)} >0 .$$
\end{enumerate}
Notice that (H1)-(H2) imply (H1') when $F = d' f$.

\subsection{Existence of traveling waves}\label{sec:exist_tw}

Under the assumptions (H1)-(H2), we have the following result on the existence of traveling waves.
\begin{theo}\label{th:wave_exists}
Assume that \textup{(H1)-(H2)} hold true, and let $c^* >0$ be the minimal traveling wave speed of the monostable semilinear equation
\begin{equation}\label{eq:scalar_semi}
\partial_t w = \partial_x^2 w + d'(w) f(w).
\end{equation}
Then equation~\eqref{eq:scalar_gen} admits a traveling wave solution with speed~$c$ if and only if $c \geq c^*$.

By traveling wave with speed~$c$, we mean a solution $w(t,x)= W(x-ct)$ such that 
%\begin{equation}\label{eq:scalar_gen_TW}
%\left( d (W) \right)'' + c W ' + f(W) =0,
%\end{equation}
%as well as
$$W(+\infty) = 0  < W (\cdot) <  W(-\infty)=  1.$$
\end{theo}
There is a wide literature on traveling wave solutions of reaction-diffusion equations, tracing back in the semilinear case (i.e. with a standard laplacian diffusion) to the seminal papers~\cite{Fisher,KPP} and slightly more but not so recently~\cite{AW75}. These works in particular ensure the existence of a minimal wave speed~$c^*$ for monostable equations including~\eqref{eq:scalar_semi} under assumptions~(H1)-(H2). The above Theorem~\ref{th:wave_exists} can be deduced from the extensive collection of results of~\cite{GildingKersner} where the general nonlinear case is dealt with. Still, for the sake of completeness, a short sketch of the proof will be given below. Beforehand, we point out that the results of~\cite{GildingKersner} include more complicated situations including degenerate diffusion, e.g. the porous medium case where $d(w) = w^m$ with $m>1$ and therefore $d' (0) = 0$. The situation at hand here, i.e. with the specific functions~$d$ and $f$ resulting from system~\eqref{eq:cross_sys_zero} and introduced in Section~\ref{sec:reduc}, excludes any such degeneracy. In particular most of the well-established methods for the existence of traveling waves may be used (phase plane analysis as in~\cite{AW75}, integral techniques as in~\cite{FifeMcLeod,GildingKersner}, or intersection number arguments in the spirit of~\cite{DGM}). Due to it being widespread especially in physics, the porous medium case has been widely studied, and we refer for instance to~\cite{Kamin} for some large-time convergence result to the traveling wave.

\begin{proof}[Proof of Theorem~\ref{th:wave_exists}]
We briefly recall from~\cite{GildingKersner} (see also the references therein) a transformation connecting both equations~\eqref{eq:scalar_gen} and~\eqref{eq:scalar_semi}.

First, let us write the two corresponding traveling wave equations, namely
\begin{equation}\label{eq:scalar_semi_TW}
 \partial_z^2 W_1 + c \partial_z W_1 + d' (W_1 ) f(W_1) = 0,
\end{equation}
and
\begin{equation} \label{eq:scalar_gen_TW}
 \partial_\xi^2 ( d (W_2)) + c \partial_\xi W_2 + f(W_2) = 0.
\end{equation}
associated respectively with~\eqref{eq:scalar_semi} and~\eqref{eq:scalar_gen}. Notice that we made the variables~$z$ and~$\xi$ explicit and distinct, even though $W_1$ and $W_2$ only depend on a single variable, to prepare for the upcoming change of variables.

Due to (H1)-(H2), equation~\eqref{eq:scalar_semi_TW} is a monostable type problem in the sense that, for all $s \in (0,1)$, 
$$d' (0) f(0) = d'(1) f(1) = 0 < d' (s) f(s).$$
Therefore, by~\cite{AW75}, there exists $c^* >0$ such that a (unique up to shift) traveling wave solution of \eqref{eq:scalar_semi} connecting $0$ and $1$ with speed~$c$ exists if and only if $c \geq c^*$.

Let us now pick any such $c \geq c^*$, and let $W_1$ be the corresponding solution of~\eqref{eq:scalar_semi_TW}. Then we introduce the change of variables
\begin{equation*}
\xi = \varphi(z) = \int_0^z d' (W_1 (s)) ds.
\end{equation*}
Due to the positivity of $d'$, the function~$\varphi$ is a $C^\infty$-diffeomorphism from $\mathbb{R}$ into itself. Therefore we may rewrite~\eqref{eq:scalar_semi_TW} in terms of this new variable~$\xi$. After an elementary computation one may eventually find that 
$$\partial_\xi^2 ( d (W_1)) + c \partial_\xi W_1 + f(W_1) = 0,$$
where with a slight abuse of notation we still denoted by~$W_1$ the function after the change of variables. Moreover, since $\varphi (\pm \infty ) = \pm \infty$, this transformation preserves the asymptotics of~$W_1$. In other words, when seen as a function of this new variable, $W_1$ solves \eqref{eq:scalar_gen_TW} and it is a traveling wave solution of~\eqref{eq:scalar_gen} with speed~$c$.

In a converse way, one can show that \eqref{eq:scalar_gen} does not have any traveling wave with speed $c < c^*$. Indeed, assume that $W_2$ solves~\eqref{eq:scalar_gen_TW} with some~$c \in \mathbb{R}$. Then, defining
$$\xi = \int_0^z d' (W_2 (s)) ds,$$
one may rewrite the traveling wave equation for $W_2$ in the variable~$z$, and get a solution of~\eqref{eq:scalar_semi_TW}. It follows that $c \geq c^*$, which completes the proof.
\end{proof}

\begin{rmk}\label{rmk:uniqueness}
The proof of Theorem~\ref{th:wave_exists} also gives uniqueness up to translation of the traveling wave solutions of \eqref{eq:scalar_gen}, since the corresponding problem \eqref{eq:scalar_semi} has this property~\cite{AW75} and the change of variables is invertible.
\end{rmk}

Now that we have characterised~$c^*$ as also the minimal wave speed of the semilinear equation~\eqref{eq:scalar_semi}, some monotonicity properties immediately follow. Indeed, a comparison principle argument shows that the speed $c^*$ is nondecreasing with respect to the growth rate function~$d' f$. This leads us to the following result, whose proof we omit.
\begin{prop}
	Under assumptions~\textup{(H1)-(H2)}, define $c^* (d,f)$ the minimal traveling wave speed of \eqref{eq:scalar_gen} (equivalently, of~\eqref{eq:scalar_semi}), where the notation is meant to highlight the dependence on the diffusion and reaction functions $d$ and~$f$.
	
	Then for any $d_1$, $d_2$, $f_1$ and $f_2$ such that
	$$d_1 '  f_1 \leq d_2 ' f_2,$$
	we have that
	$$c^* (d_1 ,f_1) \leq c^* (d_2 ,f_2).$$
	In particular, due to $f$ and $d'$ being nonnegative, we find that 
	$$d_1 ' \leq d_2 ' \ \Rightarrow \ c^* (d_1, f) \leq c^* (d_2,f),$$
	$$f_1 \leq f_2 \ \Rightarrow \ c^* (d, f_1 ) \leq c^* (d, f_2).$$
\end{prop}

\subsection{Linear determinacy}

It is well-known that the minimal speed of the monostable problem~\eqref{eq:scalar_semi} (or more generally of~\eqref{eq:scalar_semiF}) may or may not be linearly determined. More precisely, the value of $c^*$ may or may not coincide with the explicit
$$c^*_{lin} := 2 \sqrt{d'(0) f'(0)},$$% = 2 \sqrt{F'(0)}.$$
This is the smallest value of~$c$ such that there is a special solution of the form $e^{-\lambda (x-ct)}$, with $\lambda \in \mathbb{R}$, of the linearized equation at~$0$,
$$\partial_t w = \partial_x^2 w + d' (0) f '(0) w ,$$
or equivalently, of the linearization of the original problem, 
$$\partial_t w = d' (0) \partial_x^2 w + f'(0) w.$$
A sufficient (but not necessary) condition is the so-called KPP condition. We sum this up in the next proposition, whose proof is a straightforward application of classical results such as those of~\cite{AW75,KPP}.
\begin{prop}\label{prop:scalar_kpp}
Assume that
\begin{equation}\label{eq:kpp_hyp}
\forall p \in (0,1) , \quad 
d'(p) \frac{f(p)}{p}\leq d'(0) f'(0) .
\end{equation}
Then
$$c^* = c^*_{lin},$$
where $c^*$ denotes the minimal traveling wave speed for~\eqref{eq:scalar_gen}.

Otherwise, it is only known that
$$c^* \geq c^*_{lin}.$$
\end{prop}
The fact that the KPP condition~\eqref{eq:kpp_hyp} is not necessary for the linear determinacy of the minimal traveling wave speed can be observed in the following example from~\cite{HadelerRothe}, that is equation
$$\partial_t w = \partial_x^2 w + (w-w^2 )(1+aw), $$
where $a \geq 0$. Indeed, this equation is of the KPP type (i.e. $(w-w^2)(1+aw) \leq w$) only when $a \leq 1$, yet the minimal traveling speed is linearly determined (i.e. equal to $2$) if and only if $a \leq 2$. More generally, as it was more recently exhibited in~\cite{An23}, the minimal traveling wave speed of an equation of the form
\begin{equation*}%\label{eq:scalar_semiAn}
\partial_t w = \partial_x^2 w + (w- A(w))(1+ \chi A'(w)),
\end{equation*}
under some appropriate assumptions on the function $A$, is linearly determined if and only if $\chi \leq 1$, which is consistent with the previous example where $A(w) = w^2$. It turns out that the traveling wave equation resulting from our starting problem~\eqref{eq:cross_sys_zero}, that is~\eqref{eq:scalar_semi} with $d$ and $f$ as in \eqref{def:dd_tw}-\eqref{def:ff_tw}, may fall into that scope. We will come back to this in more details in the next section.

Therefore, the following result (which sums up Propositions~3.3 and~A.2 from~\cite{An23}) will be of most use to us:
\begin{prop}\label{prop:linear_determ}
	Consider equation~\eqref{eq:scalar_semiF} with the reaction term
	$$F(w) = F'(0) (w-A(w))(1 + \chi A'(w)),$$
	where $s \mapsto A(s)$ is smooth and satisfies
			$$A(0) = A' (0) = 0 = A(1) -1 , \quad \mbox{ and } \ \ \forall s \in (0,1), \ A' (s) >0  > A(s) -s.$$
	Then the minimal wave speed for~\eqref{eq:scalar_semiF} is given by
			$$ c^* = \left\{ 
			\begin{array}{ll}
				c^*_{lin} := 2 \sqrt{F'(0)} & \quad \mbox{ if } \  \chi \in [0,1],\vspace{3pt}\\
				\sqrt{F'(0)} \left( \sqrt{\chi} + \frac{1}{\sqrt{\chi}} \right) &  \quad \mbox{ if } \ \chi > 1 . \\
			\end{array}
			\right.
			$$ 
\end{prop}
\begin{proof}
	We give a sketch of the proof, for the sake of completeness. The first part of the argument is a classification of the monostable traveling wave~$W^*$ with minimal speed~$c^*$. By linearizing the traveling wave equation
	\begin{equation}\label{eq:TW_ode}
	W '' + c W'  + F(W) =0,
	\end{equation}
	around the equilibrium~$0$, one finds that necessarily 
		$$W^* (x) \sim (B x + C) e^{-\lambda x},$$
	as $x \to +\infty$, where $\lambda$ is a real root of
		\begin{equation}\label{eq:dispersion}
		\lambda^2 - c^* \lambda + F'(0) = 0,
		\end{equation}
	which in particular entails that $c^*  \geq 2 \sqrt{F'(0)}$, and $B$ may be non-zero only if $\lambda$ is a double root (i.e. $c^* = c^*_{lin} = 2\sqrt{F'(0)}$).
	
	Furthermore, it turns out that only the following three situations may occur:	
	\begin{itemize}
		\item pushed case: $c^* > c^*_{lin}$ and $\lambda$ is the larger root of~\eqref{eq:dispersion};
		\item pulled case: $c^* =c^*_{lin}$ and the traveling wave decays as $(Ax + B ) e^{-\sqrt{F'(0)} x}$ as $x \to +\infty$, with $A>0$;
		\item pushmi-pullyu: $c^* = c^*_{lin}$ and the traveling wave decays as $B e^{-\sqrt{F'(0)} x}$ as $x \to +\infty$, with $B>0$.
	\end{itemize}
	Usually, in the literature `pulled' refers to both the pulled and the pushmi-pullyu cases. Here the distinction is important because the pushmi-pullyu case allows one to spot the transition between the pulled and pushed cases.
	
	The new information that this trichotomy brings is that, when $c^* > c^*_{lin}$, then the traveling wave necessarily selects the fast decaying exponential. In short, this is a by-product of e.g. the argument of~\cite{AW75}, which consists in investigating~\eqref{eq:TW_ode} in the phase plane $(W,W')$. Then solving for a traveling wave is equivalent to finding a trajectory connecting~$(1,0)$ to~$(0,0)$ without crossing the vertical lines $\{ W= 0\}$ and $\{W =1\}$. First, one may check that $(1,0)$ is saddle point, thus there is a unique trajectory originating from it that may produce a traveling wave. This trajectory, let us denote it by $\mathcal{T}_1$, depends continuously and monotonically with respect to the speed parameter $c$. This, together with the fact that it cannot cross the horizontal axis between $W= 0$ and $W=1$, is the reason why there are traveling waves above some minimal speed $c^*$. 
	
	Next, when $c \geq c^*_{lin}$, then $(0,0)$ is a stable node. There exists a family of trajectories going to $(0,0)$ from the lower right of the phase plane. Among these, there is a lowest one, let us denote it by $\mathcal{T}_2$, which behaves either as
			$$B e^{-\sqrt{F'(0)}x},$$
			if $c = c^*_{lin}$, or as
			$$B e^{-\lambda_+ x},$$
			if $c > c^*_{lin}$, where $\lambda_+$ is the larger root of $\lambda^2 - c\lambda + F'(0) =0$. In any case, a traveling wave with speed~$c$ exists if and only if $\mathcal{T}_1$ either lies above or coincides with~$\mathcal{T}_2$ in the phase plane. Due to both trajectories depending continuously and monotonically (but in opposite ways) on~$c$, one eventually finds that, when $c= c^* > c^*_{lin}$, then $\mathcal{T}_1$ and $\mathcal{T}_2$ coincide.\smallskip
			
	All the above is true as long as $F$ satisfies the monostable (H1'). Let us now make use of its form
		\begin{equation}\label{eq:FF}
		F(w) = F'(0) (w-A(w))(1+ \chi A' (w)).
		\end{equation}
			As shown in~\cite{An23}, this is in fact a generic assumption, though in general it may not be possible to exhibit the function~$A$ explicitly.	
			
	First notice that the function $F$ depends continuously and monotonically on the parameter~$\chi$, hence so do the aforementioned trajectories $\mathcal{T}_1$ and $\mathcal{T}_2$. Therefore, fixing $c = c^*_{lin}$ in~\eqref{eq:TW_ode}, one eventually finds at most one value $\chi = \chi^*$ such that $\mathcal{T}_1$ and $\mathcal{T}_2$ coincide (i.e. the minimal wave is pushmi-pullyu), $c^* = c^*_{lin}$ for any $\chi \in (0,\chi^*)$, and $c^* > c^*_{lin}$ for any $\chi > \chi^*$. In other words, the pushmi-pullyu marks the exact transition between the pushed and pulled cases. This argument extends to any monotonic family of functions~$F$ (see also~\cite{XiaoZhou1,XiaoZhou2}). In particular, whether $\chi^*$ actually exists as a positive and finite value is not shown so far, and we will address it below.
	
	Finally, the actual interest of $F$ being of the form~\eqref{eq:FF} is that it allows us to construct some traveling waves explicitly. By looking at their decay rate and according to the trichotomy above, the formula for the speed will follow. This construction relies on rewriting the traveling wave equation in terms of the slope function
			$$\eta (w) = -W' (W^{-1} (w)).$$
	Notice that, according to the above phase plane argument, a traveling wave $W$ is necessarily decreasing so that it is indeed invertible. We then reach the differential equation
		$$\eta' \eta - c \eta + F = 0,$$
	on the interval $(0,1)$. Due to \eqref{eq:FF}, this suggests to take the ansatz
		$$ \eta : w \mapsto  \lambda (w -A(w)).$$
	Indeed, plugging this into the previous equation, we find
		$$[ \eta ' \eta - c \eta + F ] (w) = (w-A (w)) \times \left[ \lambda^2 - c \lambda + F'(0) + A' (\chi F'(0) - \lambda^2 )\right],$$
	which leads us to the simple system
		$$ \left\{ \begin{array}{l}
			\lambda^2 - c\lambda + F'(0) = 0,  \vspace{3pt} \\
			 \lambda^2 = \chi F'( 0). \end{array} \right.$$
	Assume first that $\chi =1$. Then we pick $c = 2 \sqrt{F'(0)}$, $\lambda= \sqrt{F'(0)}$ and the slope function
		$$ \eta : w \mapsto \sqrt{F'(0)} (w - A (w)).$$
	To recover a traveling wave, we solve
		$$W' = -\eta (W), \quad W(0) = \frac{1}{2},$$
	which indeed provides a solution of \eqref{eq:TW_ode} with the correct asymptotics, thanks to the positivity of $\eta$ between $0$ and $1$. Furthermore, recalling that $A (0) = A' (0) = 0$, we have that $A (w) = O (w^2)$ as $w \to 0$, and then
	$$W (x) \sim B e^{-\sqrt{F'(0)}x},$$ 
	as $x\to +\infty$, for some $B>0$. In other words, it is a pushmi-pullyu wave. According to the above discussion, the speed $c^*$ is linearly determined when $\chi \leq 1$, and non linearly determined when $\chi >1$.
	
	Next, when $\chi > 1$, we may pick
			$$ \lambda= \sqrt{\chi F'(0)}, \quad c = \frac{F'(0) + \lambda^2}{\lambda} = \sqrt{F'(0)} \left( \sqrt{\chi} + \frac{1}{\sqrt{\chi}}\right),$$
	so that $\lambda$ is the larger root of $\lambda^2 - c\lambda + F'(0) =0$. In the same way, we find a traveling wave with speed $\sqrt{F'(0)} \left( \sqrt{\chi} + \frac{1}{\sqrt{\chi}}\right)$ with fast exponential decay, which according to the above trichotomy implies that it is the traveling wave with minimal speed. This concludes the proof.			
\end{proof}

\section{Computing the minimal traveling wave speed}
\label{sec:comp}
	
In this section, we consider the special case when
$$g (u) = g_1 - g_2 u, \quad h(u) = \gamma \beta u - \delta ,$$
and further investigate the minimal traveling wave speed~$c^*$ of \eqref{eq:cross_sys_epsilon} and how it depends on the parameters. In the case when $\varepsilon = 0$, we will establish analytically an explicit formula for~$c^*$. While the issue when $\varepsilon >0$ remains open from a theoretical point of view, we will address it numerically.

	\subsection{Theoretical investigation of the minimal wave speed of~\eqref{eq:cross_sys_zero}}
	
Let us turn to the proof of Theorem~\ref{th:main_tw2}. Since statement~$(i)$ follows from Theorem~\ref{th:main_tw} and an elementary computation, here we only deal with statement~$(ii)$. 

From the previous sections, provided that $U_{initial}$ is such that
$$g_1 - g_2 U_{initial} >0,$$
we already know that there is continuum of traveling waves, with some minimal speed~$c^*$, which is also the minimal speed of a more standard semilinear and monostable scalar equation
$$\partial_t u = \partial_x^2 u + d' (u) f(u),$$
where
$$d' (u) = g_1 - g_2 U_{initial} e^{-\beta u}, \quad f(u) = \gamma U_{initial} (1 -e^{-\beta u}) - \delta u.$$
In general, this speed may or may not be linearly determined. Let us further investigate how it depends on the parameters of the problem.

In fact, it already follows from the previous section that the minimal speed $c^*$ is monotonic with respect to some of the parameters. Let us briefly summarize what we know below:
\begin{itemize}
\item the minimal speed is increasing with respect to $g_1$, which only appears in the diffusion;
\item similarly, it is decreasing with respect to $g_2$;
\item the minimal speed is increasing with respect to $\gamma$, and decreasing with respect to $\delta$; both parameters only appear in the reaction term;
\item since both diffusion and reaction are increasing with respect to~$\beta$, it follows that the minimal speed is increasing with respect to $\beta$.
\end{itemize}
In this section, we will additionally prove that:
\begin{itemize}
\item the minimal speed is nondecreasing with respect to~$U_{initial}$, which comes from the linear or nonlinear determinacy argument addressed below. 
\end{itemize}
As we will see in Figure \ref{fig:2}, the linear speed, which we denote by $c^*_{lin}$, is not monotonic with respect to $U_{initial}$. However, the linear determinacy of $c^*$ holds precisely in the parameter range where $U_{initial} \mapsto c^*_{lin}$ is increasing.\\

First we briefly recall the KPP type condition~\eqref{eq:kpp_hyp} and rewrite this  in our context. Second, we will see that we are able to find an explicit necessary and sufficient condition for the linear determinacy of the minimal traveling wave speed of~\eqref{eq:cross_sys_zero} when $g,h$ are given by~\eqref{def:example1}.
\begin{prop}
Assume that, for all $p >0$,
\begin{equation}\label{eq:kpp_cross}
(g_1 - g_2 U_{initial} e^{-\beta p} ) \left[\gamma U_{initial} \frac{1- e^{-\beta p}}{p} - \delta \right] \leq (g_1 - g_2 U_{initial}) (\gamma \beta U_{initial} - \delta) .
\end{equation}
Then
$$c^* = c^*_{lin} (U_{initial}) = 2 \sqrt{(g_1 - g_2 U_{initial}) (\gamma \beta U_{initial} - \delta) },$$
where $c^*$ denotes the minimal traveling wave speed for~\eqref{eq:cross_sys_zero}.
\end{prop}
There is no proof necessary as this is a direct application of Proposition~\ref{prop:scalar_kpp}. Notice that, for any $p>0$, we have
$$\gamma U_{initial} \frac{1- e^{-\beta p}}{p} - \delta  < \gamma \beta U_{initial} - \delta  ,$$
while 
$$g_1 - g_2 U_{initial} e^{-\beta p} >  g_1 - g_2 U_{initial}. $$
Roughly speaking, there is a conflict between the KPP reaction term, which helps satisfy the condition, and the coupled diffusion term, which acts in the opposite direction. This was to be expected, because linear determinacy means that the propagation is ``pulled'' by the leading edge of the front, where~$u$ is larger. The fact that the diffusion rate of~$v$ decreases with respect to~$u$ goes against this.

Before we proceed, let us point out that \eqref{eq:kpp_cross} can be turned into the equivalent yet more explicit single inequality
\begin{equation}\label{explicit1}
 3 \beta g_2 \gamma U_{initial} - 2 g_2 \delta - g_1 \beta \gamma \leq 0.
 \end{equation}We omit the tedious computation. A possible rough sketch is to rewrite~\eqref{eq:kpp_cross} as $\alpha (p) \leq \alpha '(0) p$, where 
$$\alpha (p) = (g_1 - g_2 U_{initial} e^{-\beta p}) (\gamma U_{initial} (1 - e^{-\beta p}) - \delta p ).$$
Inequality~\eqref{explicit1} means that $\alpha '' (0) \leq 0$, which is clearly necessary for $\alpha (p) \leq \alpha '(0)p$ to hold true. Conversely, it may be shown that $\alpha$ is concave in a neighborhood of $0$, beyond which $\alpha$ becomes negative.

\begin{prop}\label{prop:pulled}
	The minimal traveling wave speed $c^*$ of \eqref{eq:cross_sys_zero} satisfies
		$$c^* = c^*_{lin} (U_{initial}) = 2 \sqrt{(g_1-g_2 U_{initial})(\gamma \beta U_{initial} - \delta)}$$
	if and only if
		$$U_{initial} \leq U^* := \frac{1}{2} \left(\frac{g_1}{g_2} + \frac{\delta}{\gamma \beta} \right).$$
	Moreover, for all $U_{initial} > U^*$, then
		$$ c^* = c^*_{lin} (U^*)  =   g_1 \sqrt{ \frac{\gamma \beta}{g_2}}- \delta \sqrt{ \frac{g_2}{\gamma \beta}} .$$
\end{prop}
The critical value $U^*$ of $U_{initial}$ for the linear determinacy is exactly the one which produces the maximum of the function $U_{initial} \mapsto c^*_{lin}$. Thus the monotonicity of $U_{initial} \mapsto c^*$ immediately follows from this proposition.
\begin{proof}
	We again recall that $c^*$ is also the minimal traveling wave speed for the semilinear monostable equation
		\begin{equation*}
			\partial_t u = \partial_x^2 u + (g_1 - g_2 U_{initial} e^{-\beta u})(\gamma U_{initial} (1-e^{-\beta u}) - \delta u).
		\end{equation*}
	Letting 
	\begin{eqnarray*}
		F ' (0) & =  & (g_1 - g_2 U_{initial}) (\gamma \beta U_{initial} - \delta),\\
		A (u) &= & \frac{\gamma U_{initial}}{\gamma \beta U_{initial} - \delta} \left(\beta u - (1- e^{-\beta u}) \right),\\
		\chi & = & \frac{g_2}{\gamma \beta} \frac{\gamma  \beta U_{initial} - \delta}{g_1 - g_2 U_{initial}},
	\end{eqnarray*}
	we may use Proposition~\ref{prop:linear_determ} (see also~\cite{An23}). We infer that the minimal speed is linearly determined if and only if $\chi \leq 1$, which is equivalent to
		$$U_{initial} \leq \frac{1}{2} \left(\frac{g_1}{g_2} + \frac{\delta}{\gamma \beta} \right).$$
	Now assume that $U_{initial} > U^*$, so that the minimal speed is nonlinearly determined and, by Proposition~\ref{prop:linear_determ}, 
	$$c^* = \sqrt{F'(0)} \left( \sqrt{\chi} + \frac{1}{\sqrt{\chi}} \right).$$
	This rewrites here as
	$$c^* = (g_1 - g_2 U_{initial}) \sqrt{\frac{\gamma \beta}{g_2}} + (\gamma \beta U_{initial} - \delta ) \sqrt{\frac{g_2}{\gamma \beta}} = g_1 \sqrt{ \frac{\gamma \beta}{g_2}}- \delta \sqrt{ \frac{g_2}{\gamma \beta}} , $$
	which concludes the proof.
	\end{proof}

\subsection{Numerical investigation of traveling wave solutions}

So far we have, on the theoretical side, rigorously classified the traveling wave solutions of~\eqref{eq:cross_sys_zero}. 
In this section, we consider, with the aid of numerical computations, further properties of traveling waves and related solutions of~\eqref{eq:cross_sys_zero}, as well as of the more general system~\eqref{eq:cross_sys_epsilon} where both species diffuse.

Hereafter, the functions $g(u)$ and $h(u)$ are set as~\eqref{def:example1}, and the following parameter values are fixed throughout this section: 
\[
\delta=1.0, \gamma=1.0, \beta=1.5, g_1=2.0, g_2=1.0.
\]
However, our observations below have been confirmed in larger parameter ranges. Here, we note that $\frac{1}{2} \left(\frac{g_1}{g_2}+\frac{\delta}{\gamma\beta}\right)=\frac 4 3$ under this choice of parameters.

\subsubsection{Convergence to a traveling wave}

In Theorem~\ref{thm:cauchy}, we proved the existence of a unique classical solution for the Cauchy problem associated with~\eqref{eq:cross_sys_zero}. It seems reasonable to expect that, for large classes of initial data, this solution will converge to one of the traveling waves exhibited in Theorems~\ref{th:main_tw} and \ref{th:main_tw2}. Indeed, this has already been shown to be the case when there is no coupling in the diffusion~\cite{DucrotGiletti}. In this subsection, we numerically confirm this. However, we cannot deal with \eqref{eq:cross_sys_zero} on the whole line in the numerical computation. Instead, \eqref{eq:cross_sys_zero} is approximately considered on the large interval $(0,200)$ under the zero flux boundary conditions on $x=0$ and $200$. 

Figure~\ref{fig:conv-1} shows the behavior of the solution of \eqref{eq:cross_sys_zero} when the initial data is set as 
\[
u(x,0)=1.2,\quad \text{and} \quad 
v(x,0)=
\begin{cases}
0.5 & x<1,\\
0 & \text{otherwise.}
\end{cases}
\]
As time evolves, one can see that the solution of the initial value problem converges to a traveling wave solution with speed around $c=1.6$. 
In particular, it is confirmed from the right figure in Figure~\ref{fig:conv-1} that the solution propagates with an asymptotically constant profile and constant wave speed. 
\begin{figure}[htbp]
    \centering
    \includegraphics[width=140mm]{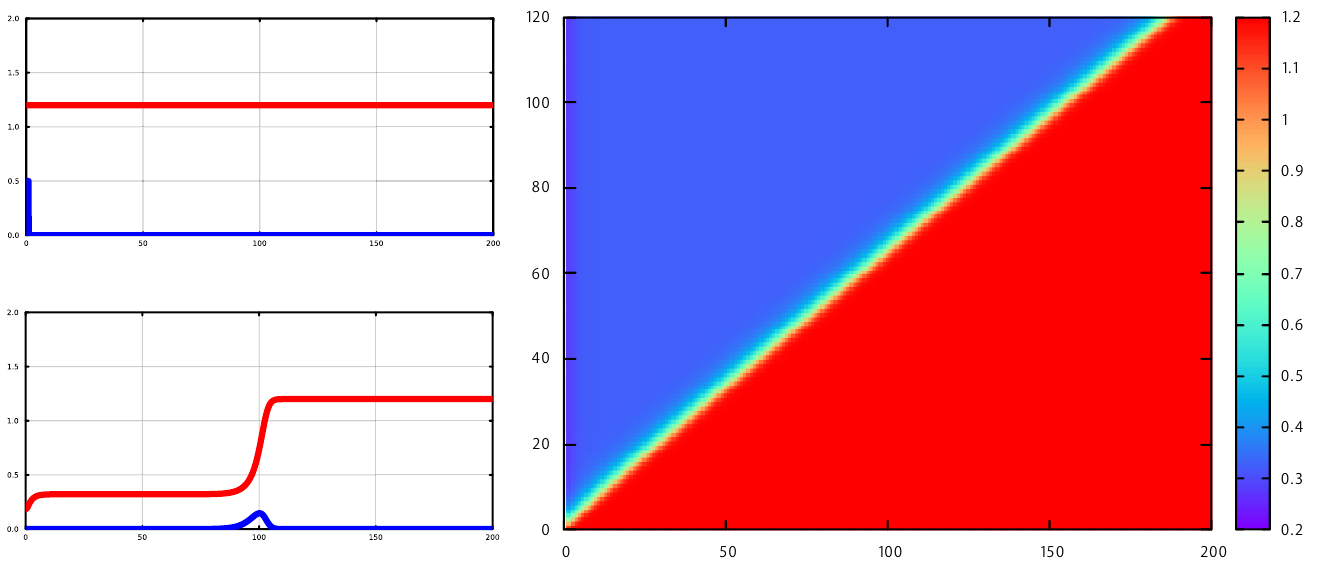}
    \caption{[Upper left] Initial data. $u(x,0)=1.2$ and $v(x,0)=0.5$ for $x<1$ and $v(x,0)=0$ for $x>1$. 
    [Lower left] A snapshot of the solution at $t=66$. In these figures, the red curve and the blue curve mean $u$ and $v$, respectively. [Right] The time evolution of the solution for $u$. The horizontal axis and the vertical axis represent space $x$ and time $t$, respectively. }
    \label{fig:conv-1}
\end{figure}
In this case, $U_{final}$ is approximately $0.321$. It seems that this traveling wave solution is the one with minimal wave speed since the value $c=1.6$ coincides with $c^*$ obtained from Theorem~\ref{th:main_tw2} in this parameter setting.

Next, we change the initial data as follows: 
\[
u(x,0)=1.2,\quad \text{and} \quad 
v(x,0)=\exp\left(\frac{-15+\sqrt{161}}{8}x\right).
\]
Our numerical computation suggests that the resulting solution also converges to a traveling wave, but a different one from the previous case, as shown in Figure~\ref{fig:conv-2}.  
\begin{figure}[htbp]
    \centering
    \includegraphics[width=140mm]{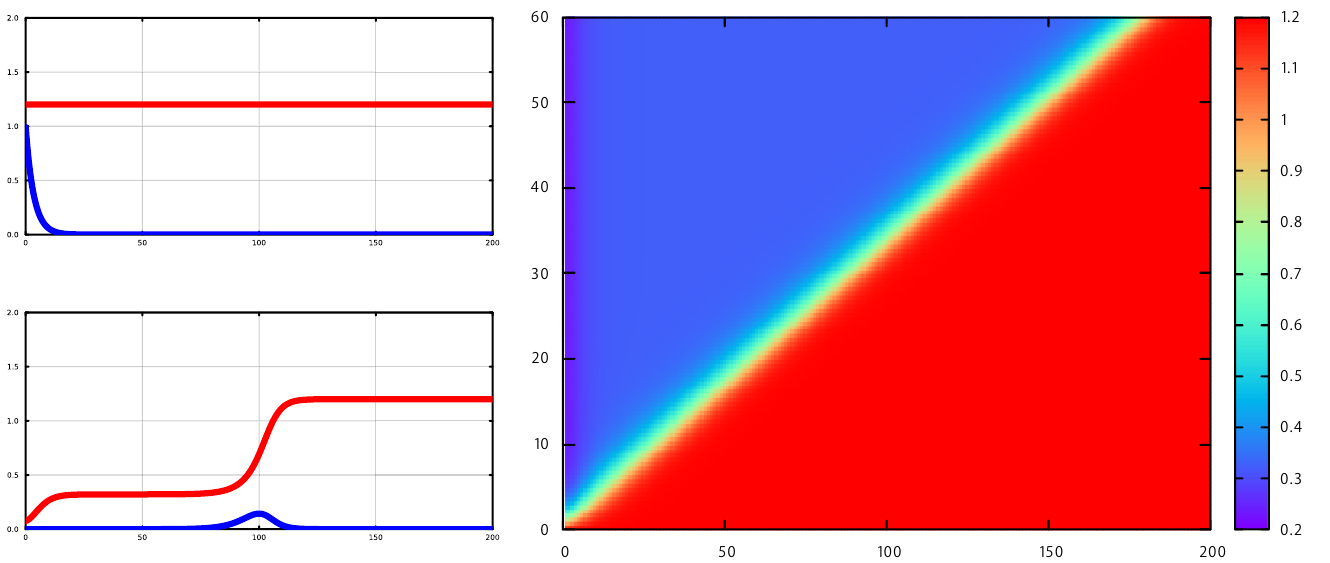}
    \caption{[Upper left] Initial data. $u(x,0)=1.2$ and $v(x,0)=\exp(\frac{-15+\sqrt{161}}{8}x)$. 
    [Lower left] A snapshot of the solution at $t=33$. In these figures, the red curve and the blue curve mean $u$ and $v$, respectively. [Right] The time evolution of the solution for $u$. The horizontal axis and the vertical axis represent space $x$ and time $t$, respectively. }
    \label{fig:conv-2}
\end{figure}
This solution propagates at around $c=3$, which is faster than the speed in Figure~\ref{fig:conv-1}. 
This implies that the asymptotic speed is determined by the exponential decay rate of the initial density of~$v$, which reflects the Fisher-KPP features of~\eqref{eq:cross_sys_zero}. 
Actually, numerical computations suggest that if the initial function for $v$ is given by $v(x,0)=e^{\eta x}$, then
\[
c=
\begin{cases}
2\sqrt{(g_1-g_2 U_{initial})(\gamma\beta U_{initial}-\delta)} & \text{ if }U_{initial}\leq \frac{1}{2}\left( \frac{g_1}{g_2}+\frac{\delta}{\gamma\beta} \right) \text{ and }\eta\leq \eta_1, \vspace{3pt}\\
g_1\sqrt{\frac{\gamma\beta}{g_2}}-\delta\sqrt{\frac{\delta}{\gamma\beta}} & \text{ if }U_{initial}> \frac{1}{2}\left( \frac{g_1}{g_2}+\frac{\delta}{\gamma\beta} \right) \text{ and }\eta\leq \eta_2, \vspace{3pt}\\
\displaystyle -\frac{\gamma\beta U_{initial}-\delta+(g_1-g_2 U_{initial})\eta^2}{\eta}  & \text{ otherwise, }
\end{cases}
\]
where $$\eta_1=-\sqrt{\frac{\gamma\beta U_{initial}-\delta}{g_1-g_2 U_{initial}}},$$ 
and $$\eta_2=-\frac{1}{2}\frac{g_1\sqrt{\gamma\beta/g_2}-\delta\sqrt{\delta/\gamma\beta}}{g_1-g_2 U_{initial}}+\frac{1}{2}\sqrt{\left(\frac{g_1\sqrt{\gamma\beta/g_2}-\delta\sqrt{\delta/\gamma\beta}}{g_1-g_2 U_{initial}}\right)^2-4\frac{\gamma\beta U_{initial}-\delta}{g_1-g_2 U_{initial}}}.$$
As a matter of fact, from our arguments in Sections~\ref{sec:reduc} and~\ref{sec:tw}, one may investigate the connection between the exponential decay~$\eta$ and the speed~$c$ of a traveling wave solution of~\eqref{eq:cross_sys_zero}, and find the same dispersion relation.

To conclude, according to numerics, the traveling wave appearing in the large-time profile of a solution strongly depends on the exponential decay rate of the initial density of the infected/predator~$v$, but not so much on the initial function of susceptible/prey~$u$. Furthermore, this property is also numerically observed for the initial value problem associated with~\eqref{eq:cross_sys_epsilon}, that is, even if the $u$-diffusion rate~$\varepsilon$ is positive.

\subsubsection{Traveling wave solutions for \eqref{eq:cross_sys_zero} and \eqref{eq:cross_sys_epsilon}}

We obtain from Theorems~\ref{th:main_tw} and \ref{th:main_tw2} that there exist traveling wave solutions with speeds $c\geq c^*(U_{initial})$, for $\frac{\delta}{\gamma\beta}<U_{initial}<\frac{g_1}{g_2}$, where
$$c^* (U_{initial}) = \left\{ 
\begin{array}{ll}
	2 \sqrt{(g_1- g_2 U_{initial}) ( \gamma \beta U_{initial} - \delta)} & \quad \mbox{ if } \  U_{initial} \leq \frac{1}{2} \left(\frac{g_1}{g_2} + \frac{\delta}{\gamma \beta} \right) ,\vspace{3pt}\\
		g_1 \sqrt{\frac{\gamma \beta}{g_2}} - \delta \sqrt{\frac{\delta}{\gamma \beta}}&  \quad \mbox{ if } \ U_{initial}   \geq  \frac{1}{2} \left(\frac{g_1}{g_2} + \frac{\delta}{\gamma \beta} \right). \\
\end{array}
\right.
$$
Figure~\ref{fig:2} illustrates the existence range of traveling wave solutions of \eqref{eq:cross_sys_zero} via numerical computations, with respect to parameters~$U_{initial}$ and~$c$. 
\begin{figure}[htbp]
    \centering
    \includegraphics[width=80mm]{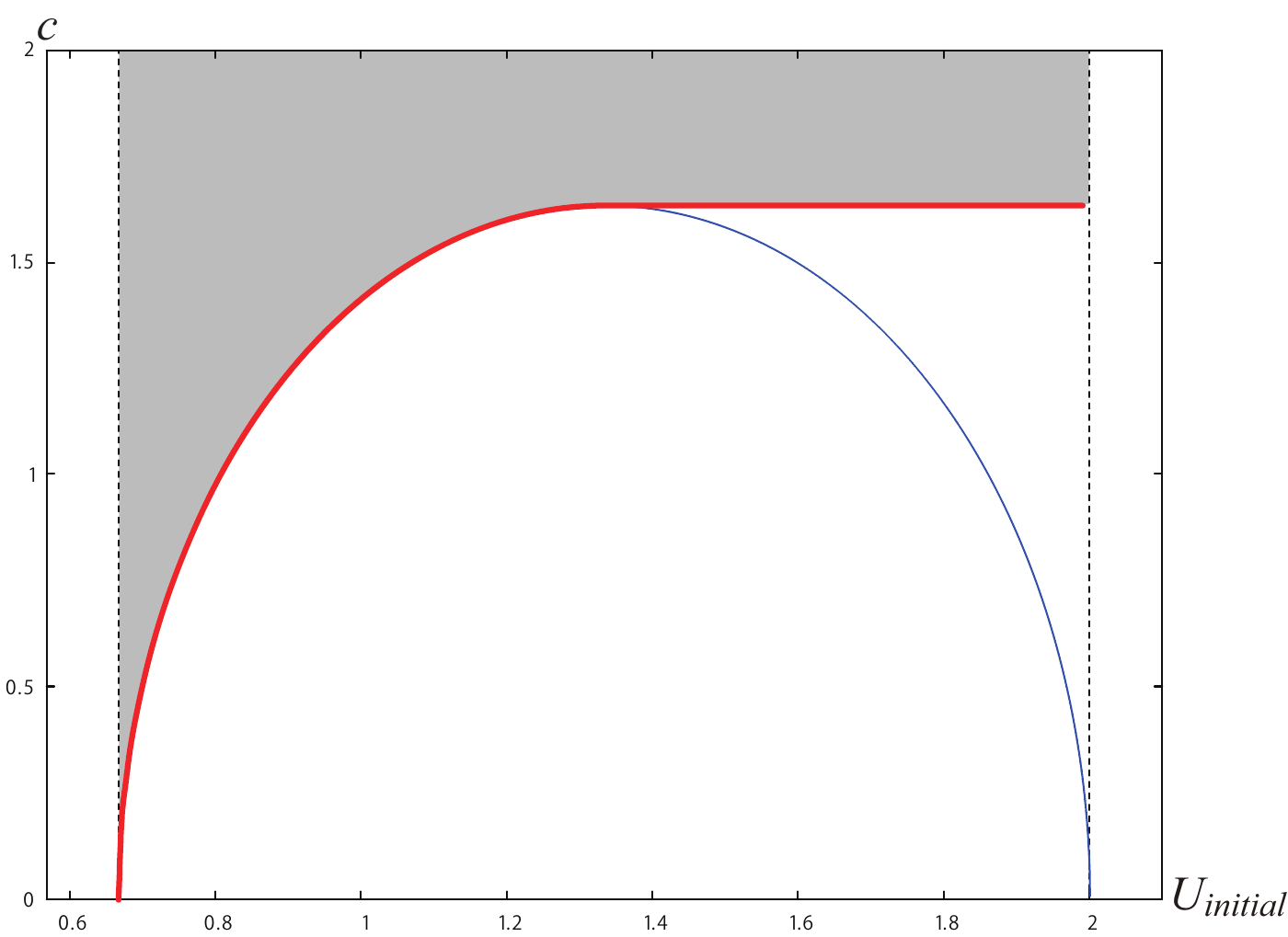}
    \caption{Existence region of traveling wave solutions when $\delta=1.0$, $\gamma=1.0$, $\beta=1.5$, $g_1=2.0$ and $g_2=1.0$, so that $\frac{1}{2} \left( \frac{g_1}{g_2} + \frac{\delta}{\gamma \beta} \right) = \frac{4}{3}$.
    The horizontal axis and the vertical axis respectively represent $U_{initial}$ and $c$. The red curve means minimum wave speed $c=c^*(U_{initial})$. There exist traveling wave solutions in the gray region.}
    \label{fig:2}
\end{figure}
The gray colored area indicates the existence region of traveling wave solutions, and the red curve represents the minimal wave speed for each~$U_{initial}$. 
Also, the blue curve is the graph of the linear speed $$c^*_{lin} =2 \sqrt{(g_1- g_2 U_{initial}) ( \gamma \beta U_{initial} - \delta)}.$$
This numerical result completely coincides with the analytical result obtained in Theorem~\ref{th:main_tw2}.

Figure~\ref{fig:profile-1} exhibits profiles of traveling wave solutions with different values of $c$. 
\begin{figure}[htbp]
    \begin{center}
    \begin{tabular}{cc}
    \includegraphics[width=60mm]{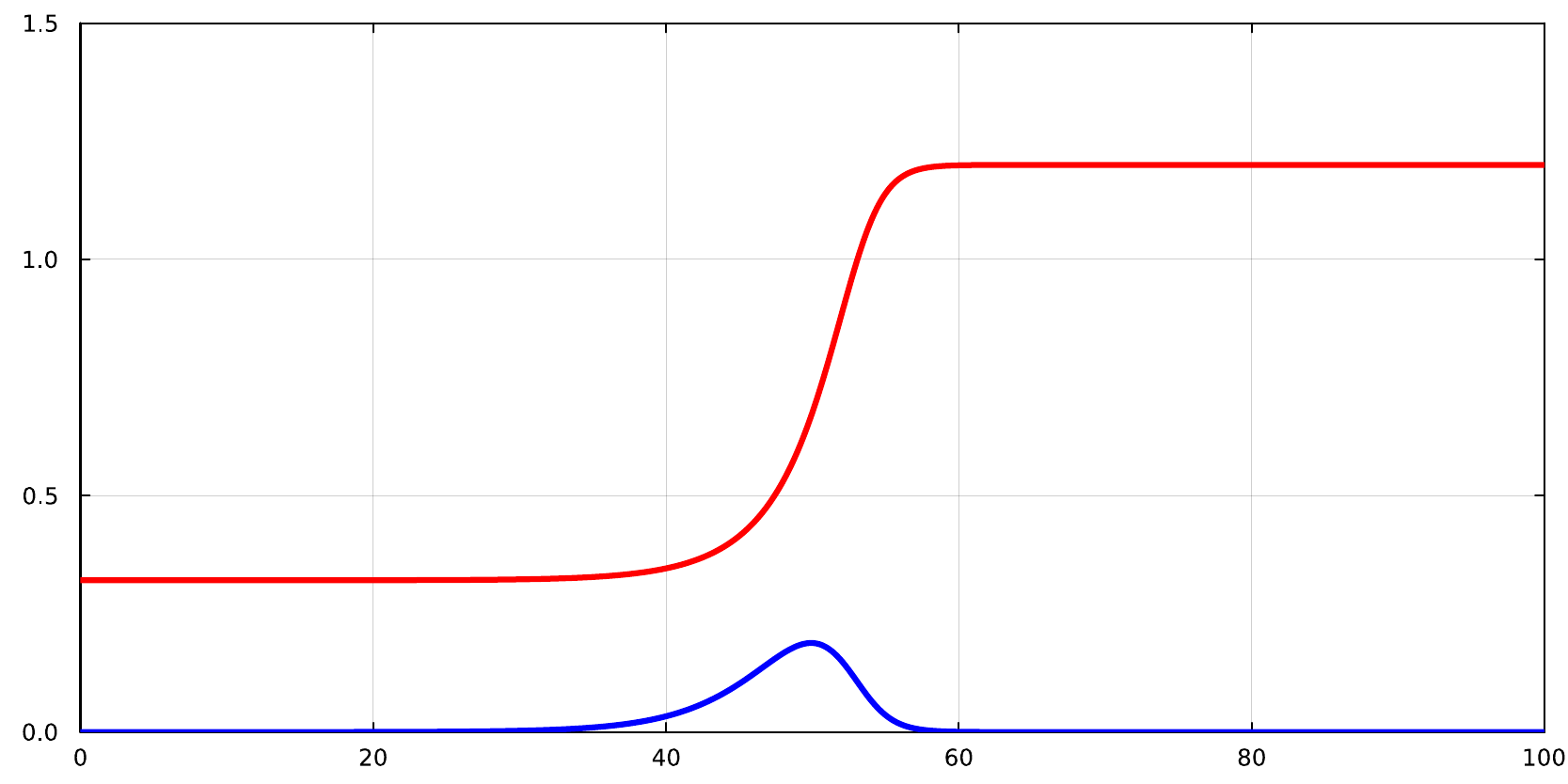}&
    \includegraphics[width=60mm]{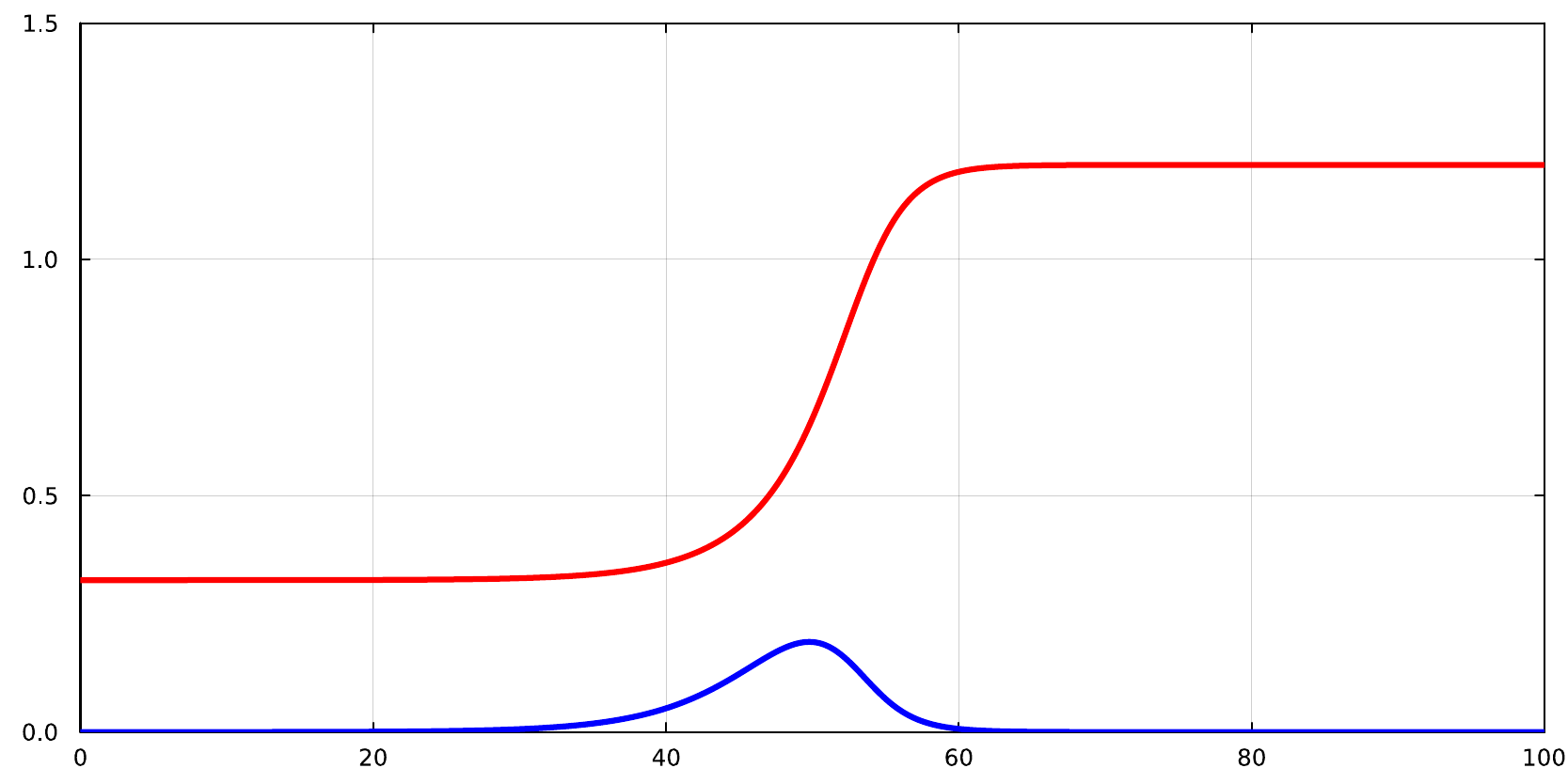}\\
    $c=1.599691$ & $c=2.0$
    \end{tabular}
    \end{center}
    \caption{Profiles of traveling wave solutions with minimal wave speeds $c=1.599691$ and $c=2.0$. The parameter values are $U_{initial}=1.2$ and the others are the same as the one in Figure~\ref{fig:2}. The resulting left limit $U_{final}=0.321084$ in both cases.}
    \label{fig:profile-1}
\end{figure}
We notice from these figures that, for each~$U_{initial}$ satisfying $\frac{\delta}{\gamma\beta}<U_{initial}<\frac{g_1}{g_2}$, then the left limit~$U_{final}$ is the same for any $c\geq c^*(U_{initial})$. 
This is consistent with Theorems~\ref{th:main_tw} and~\ref{th:main_tw2}$(i)$ where the formulas for $U_{final}$ were independent of the wave speed $c$. 
We represent the relation between~$U_{initial}$ and~$U_{final}$ in Figure~\ref{fig:3}. 
\begin{figure}[htbp]
    \centering
    \includegraphics[width=80mm]{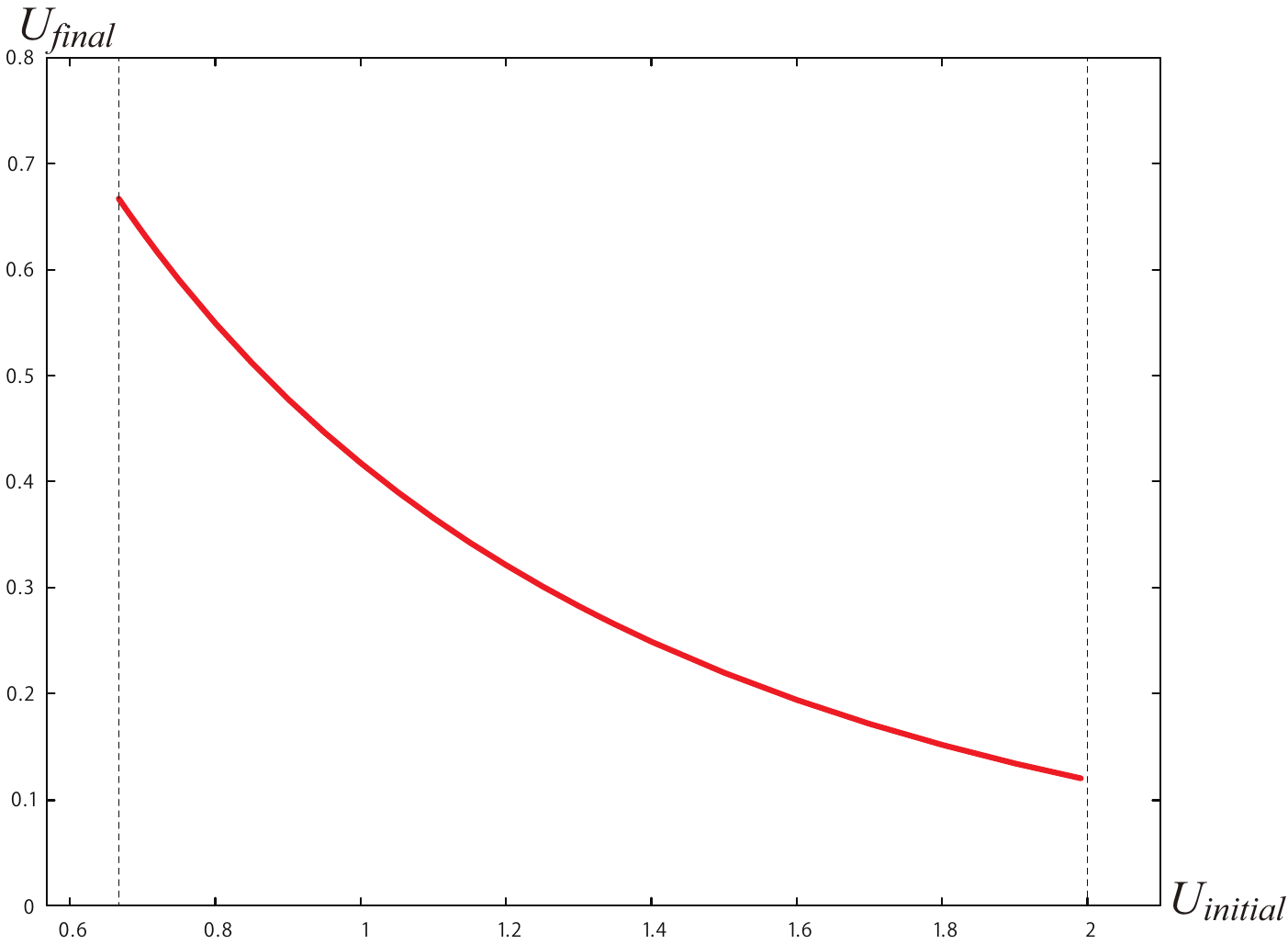}
    \caption{The relation between $U_{initial}$ and $U_{final}$. 
    The parameter values are the same as those in Figure~\ref{fig:2}. 
    The horizontal axis and the vertical axis respectively represent $U_{initial}$ and $U_{final}$. For each fixed $U_{initial}$, the left limit $U_{final}$ has the same value for traveling wave solutions with any wave speed $c\geq c^*(U_{initial})$. }
    \label{fig:3}
\end{figure}
The numerical result shows that $U_{final}$ is monotonically decreasing on $U_{initial}$. \\

Before we proceed, we briefly explain how we numerically obtained traveling wave solutions for \eqref{eq:cross_sys_zero}. These are positive solutions of~\eqref{eq:cross_sys_zero} which depend only on a moving coordinate $z=x-ct$, i.e. $\overline{u}(z)=u(x,t)$, $\overline{v}(z)=v(x,t)$, and hence solve 
\[
\left\{
\begin{aligned}
-cu'&=-\beta uv,\\
-cv'&=[(g_1-g_2 u)v]''+\gamma\beta uv-\delta v,
\end{aligned}
\right.
\quad -\infty<z<\infty, 
\]
where we already dropped the overline. 
The boundary conditions are
$(u,v)(-\infty)=(U_{final},0)$ and $(u,v)(\infty)=(U_{initial},0)$. We are looking for both $c$ and $U_{final}$ for which such a solution exists.

Letting $V=(g_1-g_2 u)v$ and $V'=\tilde{V}$, we have a system of first order ordinary differential equations
\begin{equation*}%\label{eq:numerics}
\left\{
\begin{aligned}
u'&=\frac{\beta u V}{c(g_1-g_2 u)},\\
V'&=\tilde{V},\\
\tilde{V}'&=-c\frac{\tilde{V}}{g_1-g_2 u}-\frac{g_2\beta u V^2}{(g_1-g_2 u)^3}-\frac{\gamma \beta u V -\delta V}{g_1-g_2 u},
\end{aligned}
\right.
\end{equation*}
with boundary conditions $(u,V,\tilde{V})(-\infty)=(U_{final},0,0)$ and $(u,V,\tilde{V})(\infty)=(U_{initial},0,0)$. We numerically look for a trajectory $(u,V,\tilde{V})$ connecting these two equilibria, which exists if the stable manifold of $(U_{final},0,0)$ intersects the unstable manifold of $(U_{initial},0,0)$. 
The linearized matrix around any steady state~$(\zeta,0,0)$ is given by 
\[
\begin{pmatrix}
0 & A_{12}(\zeta) & 0\\
0 & 0 & 1\\
0 & A_{32}(\zeta) & A_{33}(\zeta)
\end{pmatrix},
\]
where $A_{12}(\zeta)=\beta\zeta/(c(g_1-g_2\zeta))$, $A_{32}(\zeta)=-(\gamma\beta\zeta-\delta)/(g_1-g_2 \zeta)$ and $A_{33}(\zeta)=-c/(g_1-g_2\zeta)$. The eigenvalues of the linearized matrix are $0$ and $\lambda^{\pm}(\zeta)$, where 
\[
\lambda^{\pm}(\zeta) = \frac{1}{2}\left( -\frac{c}{g_1-g_2 \zeta}
\pm
\sqrt{\left(\frac{c}{g_1-g_2 \zeta}\right)^2-4\frac{\gamma\beta \zeta-\delta}{g_1-g_2 \zeta}}\right), 
\]
%We define 
%\[
%\begin{aligned}
%\lambda_p(\zeta) &= \frac{1}{2}\left( -\frac{c}{g_1-g_2 \zeta}+\sqrt{\left(\frac{c}{g_1-g_2 \zeta}\right)^2-4\frac{\gamma\beta \zeta-\delta}{g_1-g_2 %\zeta}}\right),\\
%\lambda_m(\zeta) &= \frac{1}{2}\left( -\frac{c}{g_1-g_2 \zeta}-\sqrt{\left(\frac{c}{g_1-g_2 \zeta}\right)^2-4\frac{\gamma\beta \zeta-\delta}{g_1-g_2 %\zeta}}\right). 
%\end{aligned}
%\]
%Note that $\lambda_p(\zeta)$, $\lambda_m(\zeta)$ and $0$ are eigenvalues of the linearized matrix of \eqref{eq:numerics} around $(\zeta,0,0)$. 
and the eigenvectors that correspond to $\lambda^{\pm}(\zeta)$ are
$$ \boldsymbol{a}^{\pm}(\zeta)=\,^t \left( 1, \frac{\lambda^{\pm}(\zeta)}{A_{12}(\zeta)}, \frac{\lambda^{\pm}(\zeta)^2}{A_{12}(\zeta)} \right).$$
%and 
%$\boldsymbol{a}_m(\zeta)=\,^t(1, \lambda_m(\zeta)/A_{12}(\zeta), \lambda_m(\zeta)^2/A_{12}(\zeta))$, where 
%$A_{12}(\zeta)=\beta\zeta/(c(g_1-g_2\zeta))$. 
Here, we note that, for any admissible pair $(c,U_{final})$, we should have $\lambda^{-}(U_{final})<0<\lambda^{+}(U_{final})$ and $\lambda^{-}(U_{initial})<\lambda^{+}(U_{initial})<0$. Thus $(U_{final},0,0)$ possesses a one-dimensional unstable manifold, and 
$(U_{initial},0,0)$ possesses a two-dimensional stable manifold.

We then use a shooting method to approximate traveling wave solutions.
Theoretically, we know that a traveling wave exists for each $c \geq c^*$; thus, $c$ is typically fixed during the shooting procedure, and only the parameter $\zeta$ corresponding to $U_{final}$ is varied.
The procedure is as follows:
\begin{enumerate}
    \item Fix $c$ and choose $\zeta$ appropriately, where $\zeta$ should be close to $U_{final}$.
    \item Compute a trajectory starting from $\boldsymbol{E}(\zeta) + \epsilon \boldsymbol{a}^{+}(\zeta)$ with sufficiently small $\epsilon > 0$, where $\boldsymbol{E}(\zeta) = {}^t(\zeta, 0, 0)$.
    The vector $\boldsymbol{E}(\zeta) + \epsilon \boldsymbol{a}^{+}(\zeta)$ is tangent to the one-dimensional unstable manifold.
    \item Check the arrival point $(\xi, 0, 0)$ of the trajectory, since it belongs to a one-parameter family of equilibria that possesses a two-dimensional stable manifold. 
    \item If $\xi$ is sufficiently close to $U_{initial}$, for example, when $|\xi - U_{initial}| < 10^{-6}$, the computation is terminated. At this point, an approximate traveling wave $(u, v, c, U_{final})$ is obtained, where $U_{final}=\zeta$.
    Otherwise, adjust $\zeta$ slightly and repeat steps~2--4.
\end{enumerate}
On the other hand, if there is no $\xi$ satisfying $|\xi - U_{initial}| < 10^{-6}$ in step~4, we conclude that no traveling wave exists for that value of $c$.
Additionally, by varying $c$ in small increments, we determine the boundary between the existence and nonexistence of traveling wave solutions, which corresponds to $c^*$. \\

Next, we consider traveling wave solutions of \eqref{eq:cross_sys_epsilon} with $\varepsilon>0$. 
In this case, we have not yet proved the existence of traveling wave solutions, thus we instead investigate it numerically.
To do that, we use the same parameter values as before except for $\varepsilon$. 
When $\varepsilon=0$, system~\eqref{eq:cross_sys_epsilon} obviously reduces to~\eqref{eq:cross_sys_zero}. 
Figure~\ref{fig:4} represents the existence range, with respect to parameters $U_{initial}$ and~$c$, of traveling wave solutions for \eqref{eq:cross_sys_epsilon} with $\varepsilon=1.0$, obtained via numerical computations.
\begin{figure}[htbp]
    \centering
    \includegraphics[width=80mm]{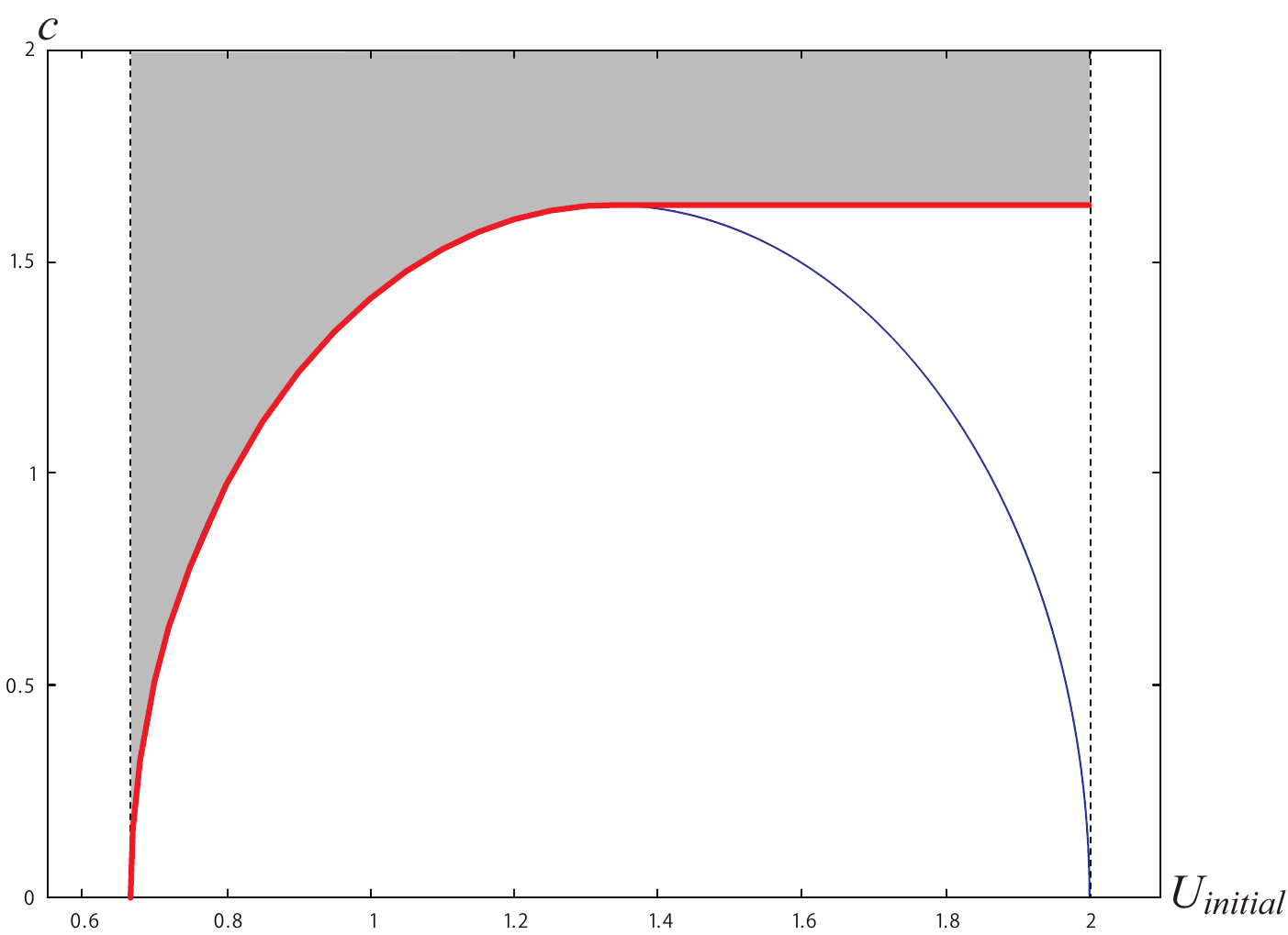}
    \caption{Existence region of traveling wave solutions when $\delta=1.0$, $\gamma=1.0$, $\beta=1.5$, $g_1=2.0$, $g_2=1.0$ and $\varepsilon=1.0$, so that $\frac{1}{2} \left( \frac{g_1}{g_2} + \frac{\delta}{\gamma \beta} \right) = \frac{4}{3}$.
    The horizontal axis and the vertical axis respectively represent $U_{initial}$ and $c$. The red curve means minimum wave speed $c=c^*(U_{initial})$. There exist traveling wave solutions in the gray region. }
    \label{fig:4}
\end{figure}
Comparing Figures~\ref{fig:2} with \ref{fig:4}, they appear to be identical. This means that the existence regions of traveling wave solutions coincide and that the minimal wave speed does not depend on~$\varepsilon$. This suggests that the diffusion rate in the equation for $u$ has no effect on the propagation speed of solutions.
While Figure~\ref{fig:4} corresponds to the single value~$\varepsilon=1.0$, we numerically confirmed that it does not change when $\varepsilon$ varies.

On the other hand, the profile of traveling wave solutions changes according to the value of $\varepsilon$ as shown in Figure~\ref{fig:5}. 
\begin{figure}[htbp]
    \begin{center}
    \begin{tabular}{cc}
    \includegraphics[width=60mm]{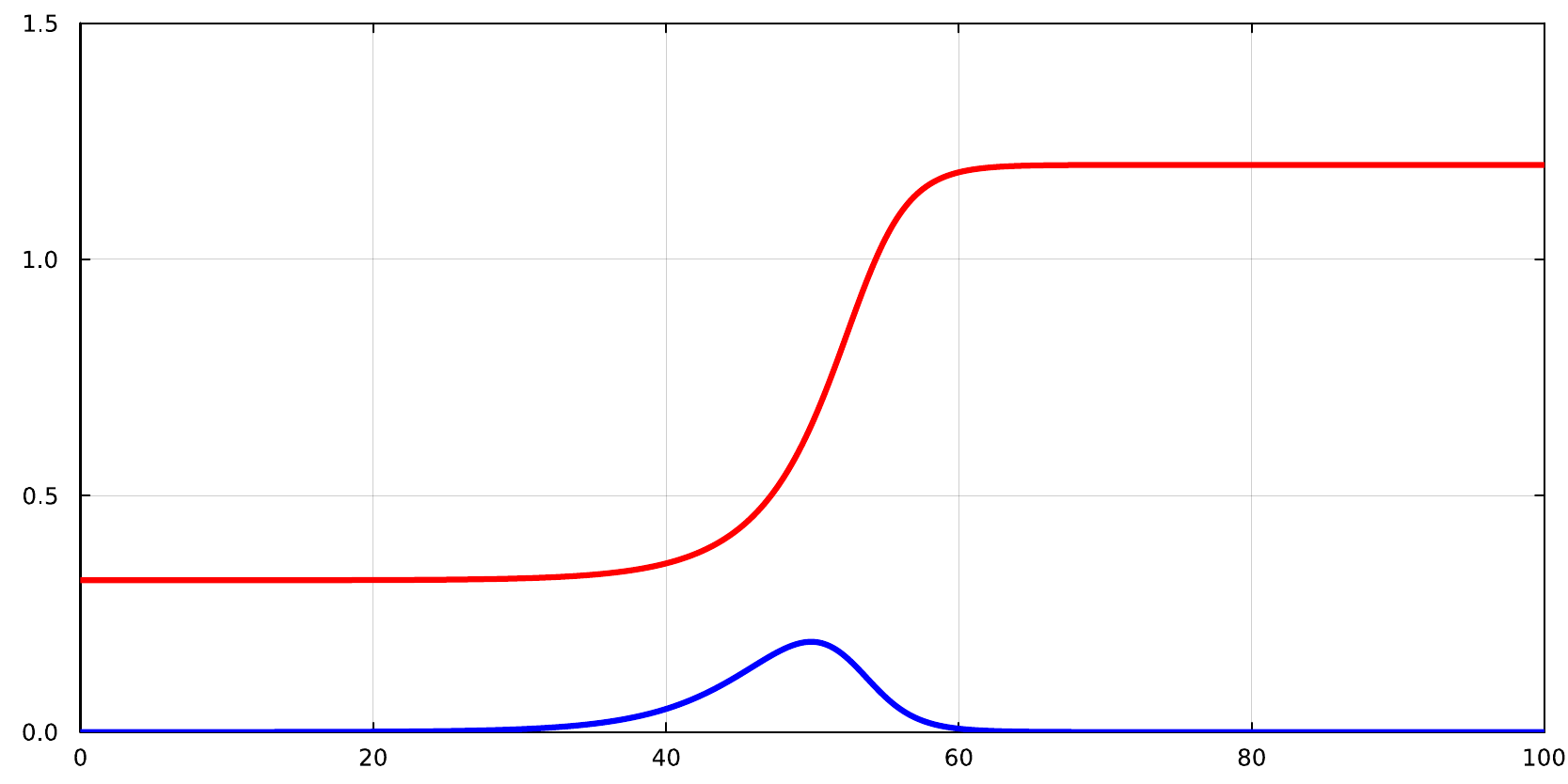}&
    \includegraphics[width=60mm]{d_001.pdf}\\
    $\varepsilon=0$ & $\varepsilon=0.01$\\
    \includegraphics[width=60mm]{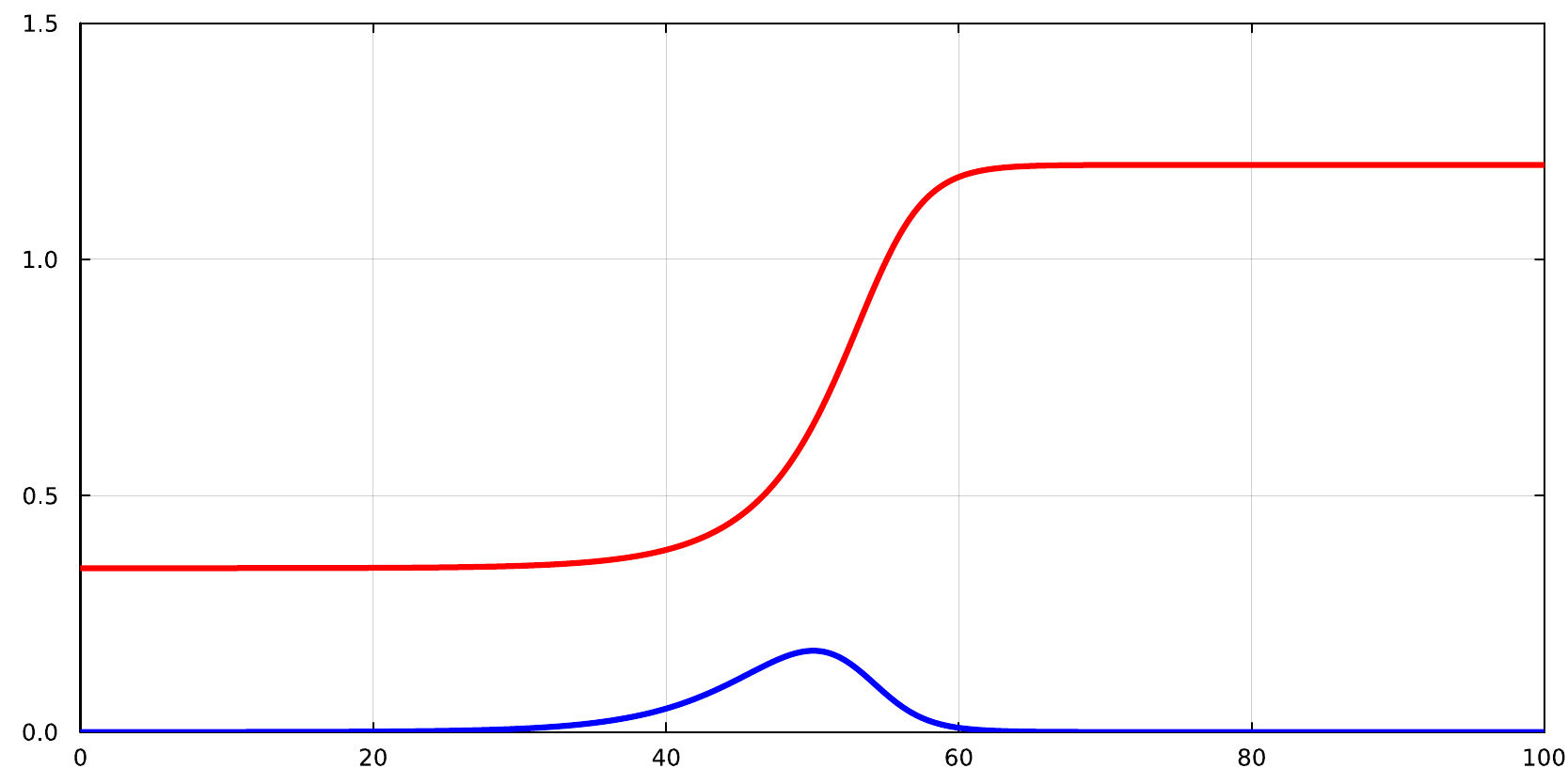}&
    \includegraphics[width=60mm]{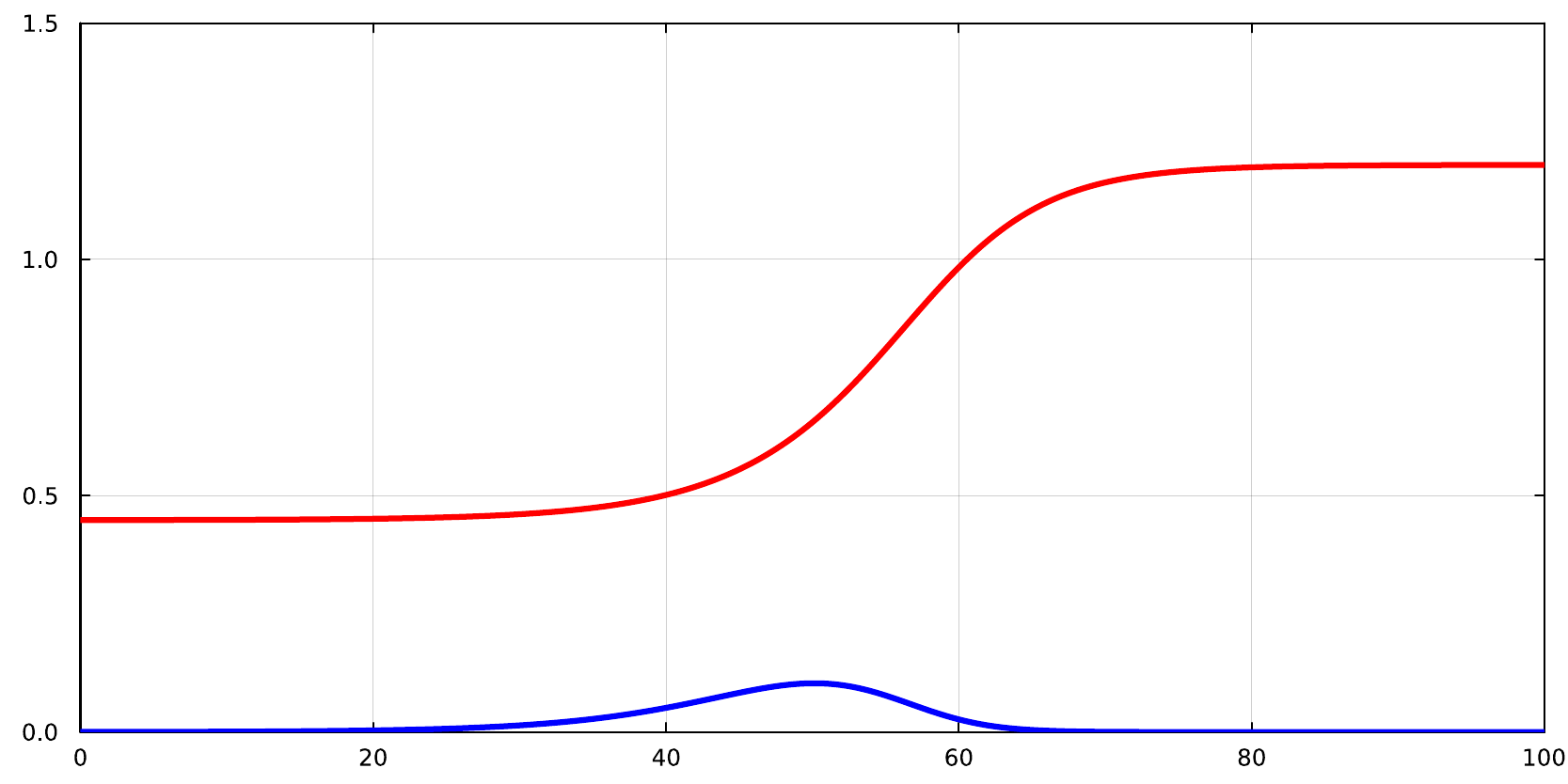}\\
    $\varepsilon=1$ & $\varepsilon=10$
    \end{tabular}
    \end{center}
    \caption{Profiles of traveling wave solutions when $\varepsilon$ varies. The parameter values are $c=2$ and $U_{initial}=1.2$, and the others are the same as those in Figure~\ref{fig:4}.}
    \label{fig:5}
\end{figure}
These figures employ the same parameter values including the speed $c=2.0$. 
As the value of $\varepsilon$ increases, the profile of the traveling wave solution becomes less steep due to the additional diffusion effect in the equation for~$u$. 
Along with this property, we can see that the value of $U_{final}$ increases. 
Figure~\ref{fig:6} illustrates the relation between the right and left limits~$U_{initial}$ and~$U_{final}$ of traveling wave solutions with the minimum wave speed for different values of $\varepsilon$. 
\begin{figure}[htbp]
    \centering
    \includegraphics[width=80mm]{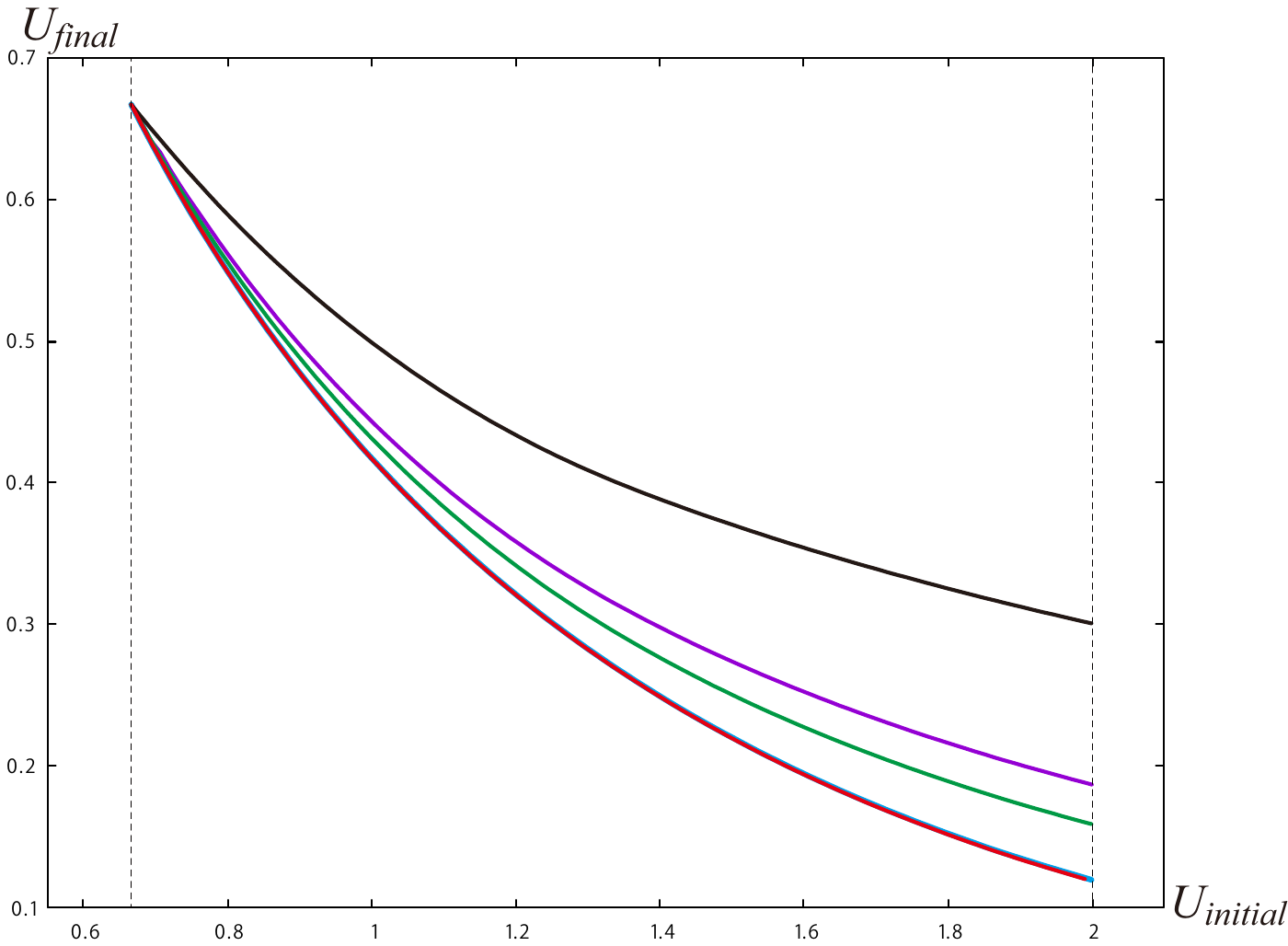}
    \caption{The relation between $U_{initial}$ and $U_{final}$ for traveling wave solutions with the minimum wave speed when the value of $\varepsilon$ varies. The colors correspond as follows: (red)$\varepsilon=0$, (blue)$\varepsilon=0.01$, (green)$\varepsilon=0.5$, (purple)$\varepsilon=1$ and (black)$\varepsilon=5$.}
    \label{fig:6}
\end{figure}
For each $\varepsilon$, the graph is monotonically decreasing with respect to~$U_{initial}$. 
Particularly, though it looks like the graphs for $\varepsilon=0$ and for $\varepsilon=0.01$ overlap, 
the curve for $\varepsilon=0.01$ is located slightly above the one for $\varepsilon=0$. This also suggests that everything is continuous with respect to the parameter~$\varepsilon$.\smallskip

To sum up, it seems that the final density~$U_{final}$ of susceptible/prey after the traveling wave has passed is decreasing with respect to the initial density~$U_{initial}$, regardless of whether this species diffuses or not. On the other hand, $U_{final}$ is increasing with respect to the diffusion rate~$\varepsilon$ of~$u$. This suggests that a higher motility rate is beneficial to the susceptible/prey~$u$, even though there is no coupling there, which means that~$u$ moves randomly and independently of the presence or not of infected/predator.

Finally, in the case of \eqref{eq:cross_sys_zero}, we have shown that~$U_{final}$ is independent of the wave speed~$c$. That is, when the parameter values are fixed except for $c$, the left limit $U_{final}$ is the same for all $c\geq c^*(U_{initial})$. However, the situation changes in the case of \eqref{eq:cross_sys_epsilon}. 
Figure~\ref{fig:7} shows the values of $U_{final}$ for several $\varepsilon$ when the wave speed $c$ varies. 
\begin{figure}[htbp]
    \centering
    \includegraphics[width=80mm]{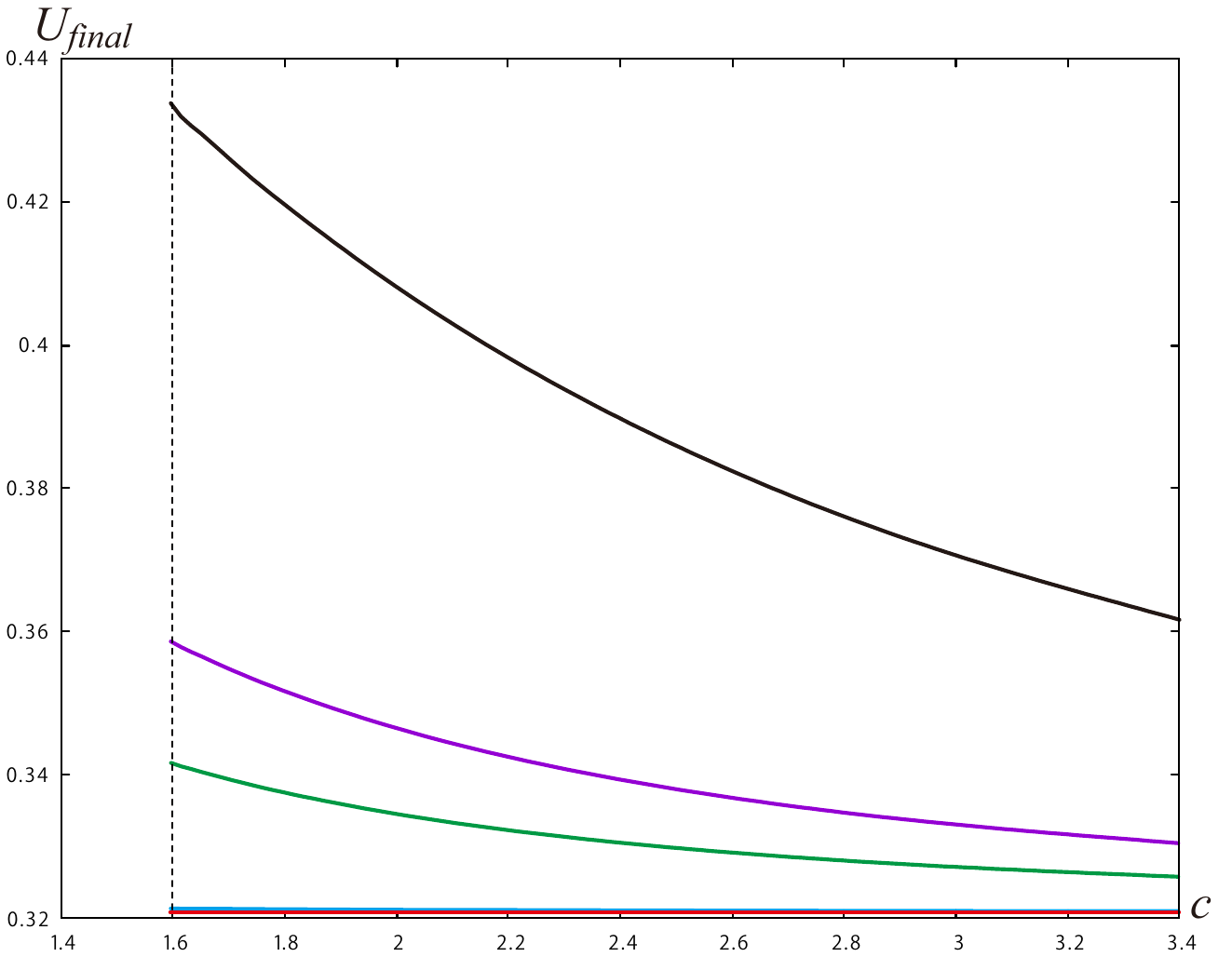}
    \caption{Dependence of $U_{final}$ on the wave speed $c$ for several $\varepsilon$ when $U_{initial}=1.2$. The colors correspond as follows: (red)$\varepsilon=0$, (blue)$\varepsilon=0.01$, (green)$\varepsilon=0.5$, (purple)$\varepsilon=1$ and (black)$\varepsilon=5$.}
    \label{fig:7}
\end{figure}
For $\varepsilon=0$, the same final density~$U_{final}$ appears for all~$c$, consistently with Theorem~\ref{th:main_tw2}$(i)$ (see also Figure~\ref{fig:3}). 
When $\varepsilon>0$, numerical computations suggest that $U_{final}$ is monotonically decreasing in $c$. This last observation is an important distinction between systems~\eqref{eq:cross_sys_epsilon} and~\eqref{eq:cross_sys_zero}.

\subsubsection{On different forms of the function~$g(u)$}

As far as system~\eqref{eq:cross_sys_zero} is concerned, we proved the existence of traveling waves for large choices of the function~$g$. Thus it seems natural to investigate, at least numerically, how the propagation dynamics depend on this choice, in both cases of systems~\eqref{eq:cross_sys_zero} and~\eqref{eq:cross_sys_epsilon}.

We first choose the following function form as a different $g(u)$: 
\[
g(u)=2\left( e^{-\frac{g_1}{g_2}u}-e^{-\left(\frac{g_1}{g_2}\right)^2}\right), 
\]
where we set $g_1=2$ and $g_2=1$. 
This function is convex and decreasing. 
Moreover, the same condition as before, that is $\frac{\delta}{\gamma \beta}<U_{initial}<\frac{g_1}{g_2}$, may ensure the existence of traveling wave solutions.
%In this case, we numerically investigate how the structures of traveling wave solutions of \eqref{eq:cross_sys_zero} and \eqref{eq:cross_sys_epsilon} vary compered with the case of $g(u)=g_1-g_2 u$. 

Figure~\ref{fig:8} shows the minimum wave speeds of traveling wave solutions of \eqref{eq:cross_sys_epsilon} with $\varepsilon=0$, $0.1$ and $1$ when $U_{initial}$ varies. 
\begin{figure}[htbp]
    \centering
    \includegraphics[width=80mm]{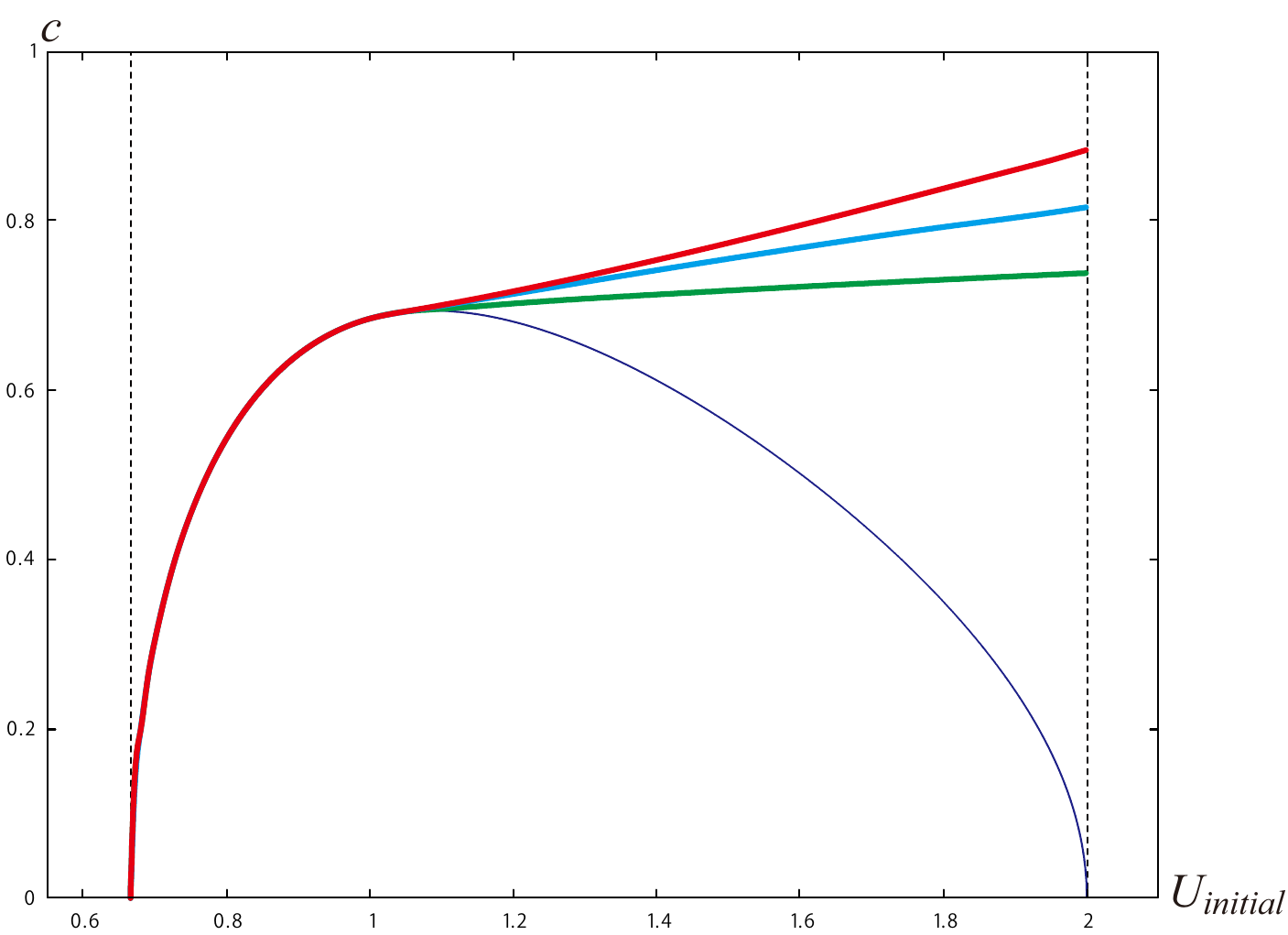}
    \caption{Minimum wave speed of traveling wave solutions of \eqref{eq:cross_sys_zero} and \eqref{eq:cross_sys_epsilon} when  $g(u)=2(e^{-2u}-e^{-4})$. The colors correspond as follows: (red)$\varepsilon=0$, (blue)$\varepsilon=0.1$, (green)$\varepsilon=1$. }
    \label{fig:8}
\end{figure}
One can see from this figure that the minimum wave speed appears not to change with respect to $\varepsilon$ for smaller values of $U_{initial}$, that is up to around~$U_{initial}=1$. Incidentally, the curves of the minimum wave speeds coincide with the curve of the linear speed $c^*_{lin}=2\sqrt{g(U_{initial})h(U_{initial})}$ (which we recall is independent of~$\varepsilon$, and is also plotted with thin dark blue on Figure~\ref{fig:8}) on the same range of values for~$U_{initial}$. On the other hand, the minimum wave speed is strictly larger than the linear speed $c_{lin}^*$, and also it is decreasing in~$\varepsilon$, if $U_{initial}$ is roughly larger than~$1$. 
This contrasts with the previous case of $g(u)=g_1-g_2 u$, where the minimum wave speed did not change with respect to $\varepsilon$ regardless of~$U_{initial}$ (see Figures~\ref{fig:2} and \ref{fig:4}). Still, for both those choices of the function~$g$, the minimum wave speed is monotonically nondecreasing in $U_{initial}$. 
Furthermore, we numerically confirm that traveling wave solutions exist in the region above each curve of minimum wave speed in Figure~\ref{fig:8}. 

\begin{figure}[htbp]
    \centering
    \begin{tabular}{c c}
    \includegraphics[width=68mm]{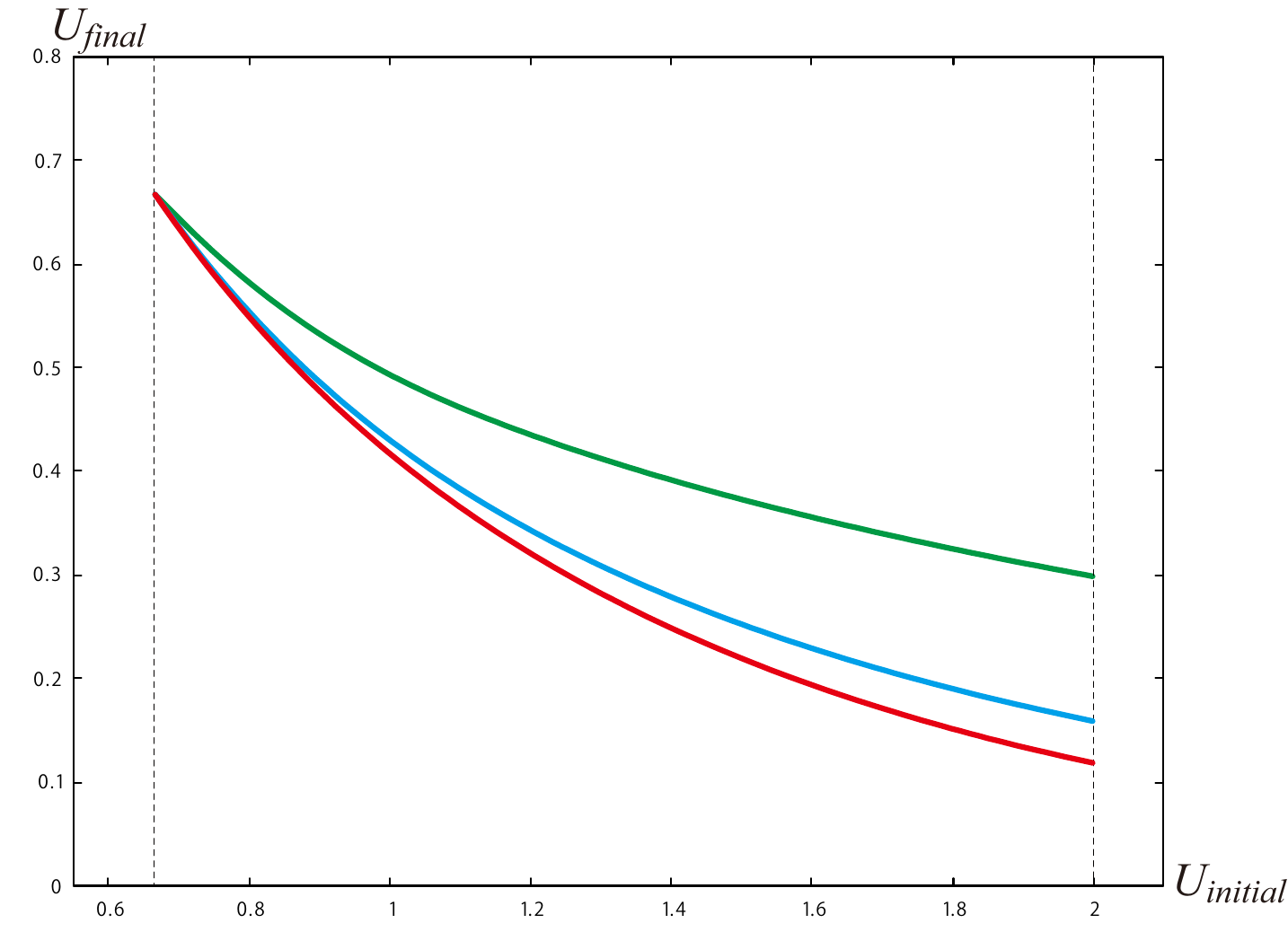}&
    \includegraphics[width=64mm]{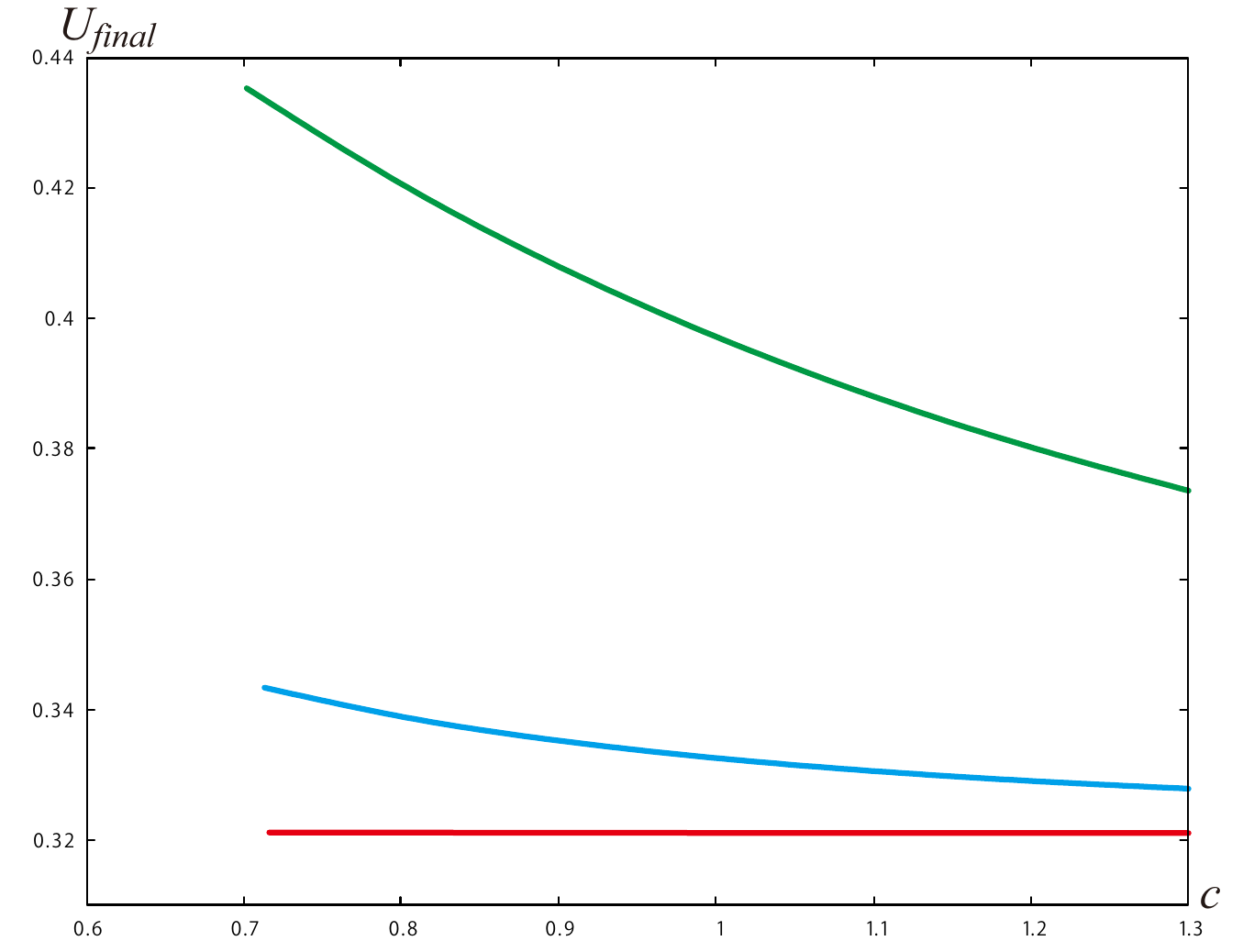}\\
    (a) & (b) 
    \end{tabular}
    \caption{
    (a) A relation between $U_{initial}$ and $U_{final}$ for traveling wave solutions with minimum wave speed. 
    (b) A relation between $c$ and $U_{final}$ when $U_{initial}=1.2$. 
    The colors correspond as follows: (red)$\varepsilon=0$, (blue)$\varepsilon=0.1$, (green)$\varepsilon=1$.}
    \label{fig:9}
\end{figure}
Next, we discuss the behavior of $U_{final}$. 
Figure~\ref{fig:9} plots relations between~$U_{initial}$ and~$U_{final}$, and between $c$ and $U_{final}$, for $\varepsilon=0$, $0.1$ and $1$. 
Figure~\ref{fig:9}(a) illustrates that~$U_{final}$ is monotonically decreasing with respect to $U_{initial}$, and Figure~\ref{fig:9}(b) shows that~$U_{final}$ has the same value for any $c$ when $\varepsilon=0$, but $U_{final}$ declines according to $c$ when $\varepsilon>0$. 
These properties are similar to the case $g(u)=g_1-g_2 u$ (see Figures~\ref{fig:6} and~\ref{fig:7}). 
Note that the left edge of each curve in Figure~\ref{fig:9}(b) corresponds to the traveling wave solution with minimum wave speed.\smallskip 

Finally, we set the function~$g$ as
\[
g(u)=\frac 1 3\left( e^{g_1/g_2}-e^u \right), 
\]
which is concave and decreasing. 
\begin{figure}[htbp]
    \centering
    \includegraphics[width=80mm]{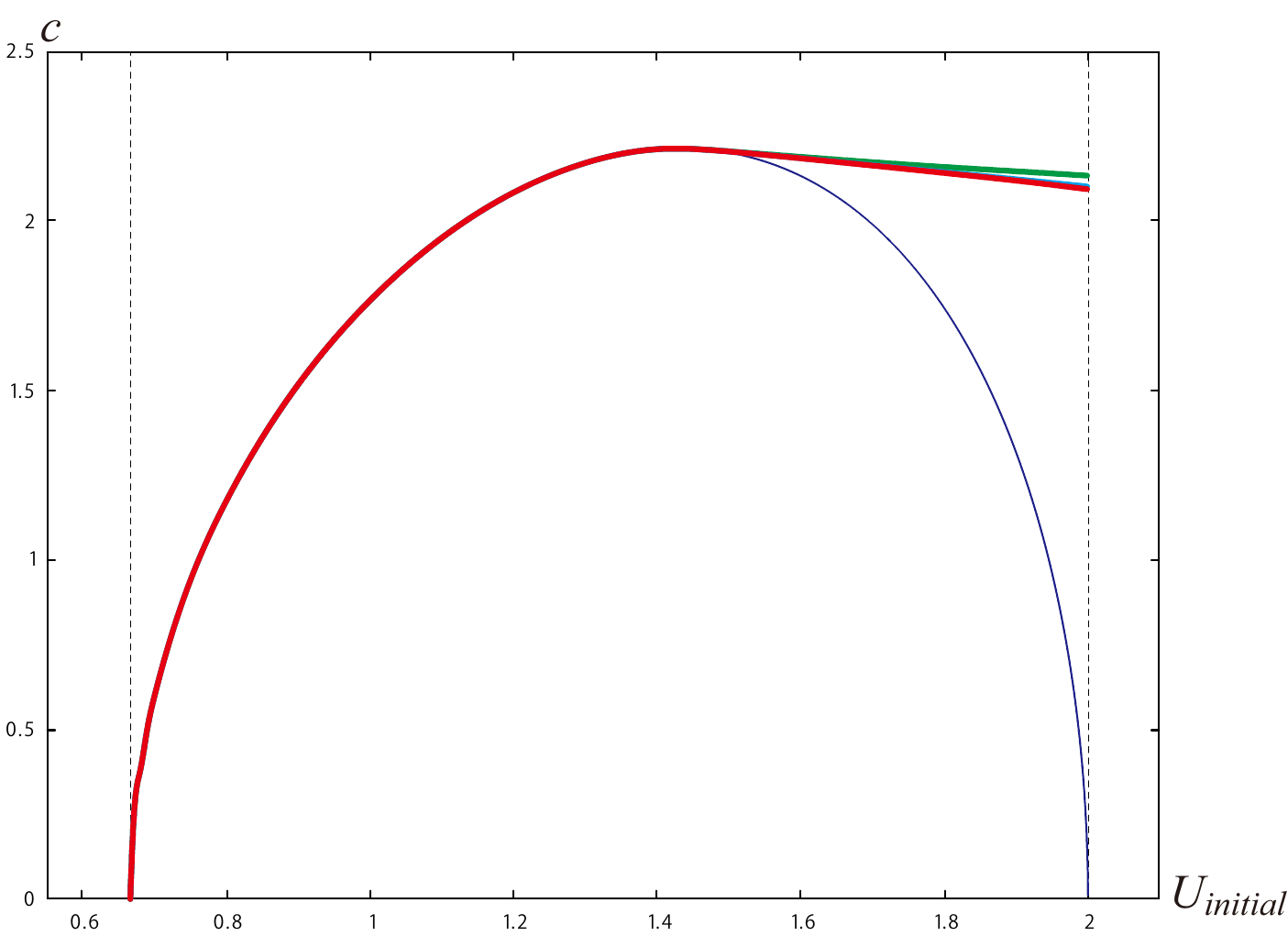}
    \caption{Minimum wave speed of traveling wave solutions of \eqref{eq:cross_sys_zero} and \eqref{eq:cross_sys_epsilon} when  $g(u)=\left( e^{g_1/g_2}-e^u \right)/3$. The colors correspond as follows: (red)$\varepsilon=0$, (blue)$\varepsilon=0.1$, (green)$\varepsilon=1$. }
    \label{fig:10}
\end{figure}
In this case, Figure~\ref{fig:10} shows the relation between~$U_{initial}$ and the minimum wave speed~$c$. 
Unlike in the previous cases (see Figures~\ref{fig:2}, \ref{fig:4} and \ref{fig:8}), the minimum wave speed $c$ possesses a peak around $U_{initial}=1.4$. On the left of this peak, the curves coincide with the graph of the linear speed $c_{lin}^*=2\sqrt{g(U_{initial})h(U_{initial})}$, and in particular the minimum wave speed does not depend on $\varepsilon$. On the other hand, one can see from Figure~\ref{fig:10} that for a fixed $U_{initial}$ on the right of the peak, the minimum wave speed increases as the value of $\varepsilon$ becomes larger. 
This tendency is opposite to the case $g(u)=2(e^{-g_1 u/g_2}-e^{-\left(g_1/g_2\right)^2})$ (see Figure~\ref{fig:8}), which appears to come from the different concavity/convexity properties of both these forms of~$g(u)$. Moreover, we again checked numerically that traveling wave solutions exist in the region above each curve. \smallskip

To conclude, we have discussed here the structure of traveling wave solutions for three types of $g(u)$. 
The first case is $g(u)=g_1-g_2 u$ which is a linear decreasing function, the second is $g(u)=2(e^{-g_1 u/g_2}-e^{-\left(g_1/g_2\right)^2})$ which is convex and decreasing, and the third one is $g(u)=\frac 1 3\left( e^{g_1/g_2}-e^u \right)$ which is concave and decreasing.
Those three cases reveal that the structure of traveling wave solutions may differ depending on the choice of the function~$g$. From these numerical results, we expect that the minimum traveling wave speed of~\eqref{eq:cross_sys_zero} is faster than that of~\eqref{eq:cross_sys_epsilon} if $g$ is convex and decreasing. 
On the other hand, 
the minimum traveling wave speed of \eqref{eq:cross_sys_zero} is slower if $g$ is concave and decreasing.
In the critical case when $g$ is affine, the minimum traveling wave speeds of \eqref{eq:cross_sys_zero} and \eqref{eq:cross_sys_epsilon} coincide. 
These are our conjectures on minimum traveling wave speeds of \eqref{eq:cross_sys_zero} and \eqref{eq:cross_sys_epsilon}.

\section*{Acknowledgments}
T. G. was supported by the Agence Nationale de la Recherche projects Indyana (ANR-21-CE40-0008) and Reach (ANR-23-CE40-0023-01). H. I. was supported by JSPS KAKENHI Grant Number 25K07120. H. M. was supported by JSPS KAKENHI Grant Number 23K03216.
The three authors acknowledge the support of CNRS through the International Research Network ReaDiNet.

\section*{Data availability} Data sharing does not apply to this article because no datasets were generated or analyzed during the current study.

\section*{Conflict of interest} The authors declare that they have no Conflict of interest. All authors satisfy ethical responsibilities.

\bibliographystyle{plain}
\bibliography{biblio}

%\printbibliography%Prints bibliography

\end{document}